\documentclass[12pt]{amsart}

\numberwithin{equation}{section}

\usepackage{amsmath,amsthm,amsfonts,eucal,}
\usepackage{xypic}
\usepackage{graphicx}
\usepackage{amssymb}
\usepackage{mathtools}
\usepackage{url}
\usepackage{xcolor}

\def\cb{{\mathcal B}}

\def\cf{{\mathcal F}}

\def\ch{{\mathcal H}}

\def\ck{{\mathcal K}}

\def\cs{{\mathcal S}}
\def\ct{{\mathcal T}}

\def\ga{{\mathfrak A}} 
\def\gb{{\mathfrak B}}

\def\bc{{\mathbb C}}

\def\bj{{\mathbb J}}

\def\bm{{\mathbb M}}
\def\bn{{\mathbb N}}

\def\bp{{\mathbb P}}
\def\br{{\mathbb R}}
\def\bt{{\mathbb T}}

\def\bz{{\mathbb Z}}

\def\a{\alpha}

\def\g{\gamma}

\def\l{\lambda}

\def\s{\sigma}

\def\om{\omega} \def\Om{\Omega}

\newtheorem{thm}{Theorem}[section]
\newtheorem{example}[thm]{Example}
\newtheorem{lem}[thm]{Lemma}
\newtheorem{cor}[thm]{Corollary}
\newtheorem{prop}[thm]{Proposition}

\theoremstyle{definition}
\newtheorem{rem}[thm]{Remark}
\newtheorem{defin}[thm]{Definition}

\def\min{\mathop{\rm min}}

\def\aut{\mathop{\rm Aut}}

\begin{document}
\title[A Dynamical Approach to Non-commutative de Finetti Theory]
{A Dynamical Approach to Non-commutative de Finetti Theory}

\author{Simone Del Vecchio}
\address{Simone Del Vecchio\\
Dipartimento di Matematica\\
Universit\`{a} degli studi di Bari\\
Via E. Orabona, 4, 70125 Bari, Italy}
\email{\texttt{simone.delvecchio@uniba.it}}

\author{Stefano Rossi}
\address{Stefano Rossi\\
Dipartimento di Matematica\\
Universit\`{a} degli studi di Bari\\
Via E. Orabona, 4, 70125 Bari, Italy}
\email{\texttt{stefano.rossi@uniba.it}}

\begin{abstract}
We develop a dynamical framework for non-commutative de Finetti theory. 
We first establish the non-commutative Hewitt-Savage $0$-$1$ law for quantum stochastic processes. We identify the factorization condition of the distribution of a spreadable process
which characterizes tail-triviality, which is also characterized dynamically in terms of a certain attractivity property of the distribution. These three equivalent ergodic conditions identify a distinguished level in the non-commutative hierarchy of ergodic properties, which all collapse to ergodicity in the classical probability setting. \\
To pass from the ergodic to the general case, we construct the conditional expectation onto the tail algebra for a spreadable process, in the GNS representation of the distribution of the canonical bilateral extension of the process. In this representation we establish the non-commutative Olshen Theorem identifying the tail algebra with the stationary algebra, and with the exchangeable algebra when the process is exchangeable. The resulting conditional expectation in particular inherits the same type of factorization property as distributions satisfying the Hewitt-Savage $0$-$1$ law. This factorization strengthens conditional independence conditions arising in previous literature, while collapsing to the same notion in the classical case. \\
This dynamical viewpoint leads to the identification of the minimal distributional symmetry underlying non-commutative de Finetti Theory, which we call \emph{weak spreadability}, and which in the classical setting is equivalent to exchangeability. We prove that a stationary process is weakly spreadable if and only if its tail algebra admits a unique normal conditional expectation satisfying Hewitt-Savage type factorization, thereby establishing a general non-commutative de Finetti Theorem.


\end{abstract}

\maketitle

\section{Introduction}
Two fundamental structural results in the classical theory of exchangeable stochastic processes are the theorems of De Finetti and Hewitt-Savage. De Finetti's theorem identifies exchangeable processes  with those that are conditionally independent and identically distributed with respect to the tail $\sigma$-algebra, or equivalently, with mixtures of independent and identically distributed processes. Complementing this result, the Hewitt-Savage $0$-$1$ law characterizes the ergodic \emph{i.i.d.} case through the triviality of the tail $\sigma$-algebra. Together with the Olshen Theorem, which identifies the tail $\sigma$-algebra with both the stationary and exchangeable $\sigma$-algebras, these results provide an accomplished description of exchangeable processes.

The extension of this picture to non-commutative probability is considerably more subtle. In contrast with the classical setting, where there is a unique notion of independence, several inequivalent notions of independence emerged in non-commutative probability following Voiculescu's introduction of free probability \cite{VDN92}. Among them, tensor, free, Boolean, monotone and anti-monotone constitute Muraki's five natural notions of independence \cite{Mur}. Other forms have also been introduced, see e.g. \cite{BLS,Vbifree}.

The existence of several notions of independence naturally raises the question of how de Finetti theory should be extended beyond the classical framework.
The first major breakthrough in this direction was achieved by K\"{o}stler and Speicher \cite{KS}, who introduced the notion of quantum exchangeability and proved a free de Finetti theorem. Quantum exchangeability replaces classical exchangeability by the invariance under the action of Wang's quantum permutation groups, which is a stronger symmetry, and characterizes free conditional i.i.d. with respect to the tail algebra. Their work established a new paradigm, in which a prescribed notion of conditional i.i.d. is characterized by a suitable quantum symmetry.

Analogous de Finetti theorems were established for the remaining Muraki's natural notions of independence by Liu in \cite{Liu,LiuBM,Liu19} and, more generally, for a variety of quantum symmetries associated with so-called easy quantum groups \cite{BCS, BCFMW, Liu19, Liu25}. 

A different perspective was developed by K\"{o}stler in \cite{K}. Motivated by the theorem of Ryll-Nardzewski, identifying spreadability with exchangeability in the classical setting, he proposed spreadability as the fundamental distributional symmetry in the non-commutative framework. It is worth mentioning that spreadability is also the distributional symmetry naturally shared by the various non-commutative frameworks described above. Rather than starting from a prescribed notion of independence, spreadability itself is shown to imply a weak form of conditional independence, which in K\"{o}stler's terminology goes under the name of full conditional independence with respect to the tail algebra. Since in the classical setting conditional i.i.d. and spreadability/exchangeability determine each other, it is natural to ask whether this remains true in the non-commutative framework. The answer is negative: as shown in  \cite{GK08, K}, stationarity together with full conditional independence do not suffice to reconstruct spreadability.

%

The present work addresses precisely this issue. Our guiding principle is that the appropriate notion of conditional independence associated with the fundamental distributional symmetry should first be identified at the ergodic level through a suitable non-commutative analogue of the Hewitt-Savage $0$-$1$ law. Accordingly, our first main result establishes such a theorem for spreadable processes, Theorem \ref{01lawspread}. We prove that tail triviality is equivalent to a three-block factorization property of the distribution, which we call the forward-decoupling property, and, equivalently, to a dynamical attractivity property. 

This characterization naturally motivates a hierarchy of ergodic properties. The classical theory admits several equivalent characterizations of ergodicity, whereas in the non-commutative setting these are known to be no longer equivalent. Our analysis organizes them into such a hierarchy, Theorem \ref{ergohierspread}, and identifies the precise level occupied by the non-commutative Hewitt-Savage law. In particular, the factorization corresponding to the Hewitt-Savage level is genuinely stronger than the natural non-commutative analogue of the classical product factorization of moments. Furthermore, we show that the faithfulness assumptions commonly adopted in the literature (e.g. in \cite{BCS, GK08, K,KS}) force all levels of this hierarchy to coincide, motivating an approach that does not rely on such assumptions.

The next step is to construct the canonical conditional expectation onto the tail algebra. Such a construction was obtained by K\"{o}stler under faithfulness assumptions \cite{K} and, by Liu for bilateral processes \cite{LiuBM}. On the other hand, it is known that, without faithfulness, the GNS representation of the original (unilateral) process need not admit a stationary normal conditional expectation onto the tail algebra \cite{DKW,LiuBM}. This leads naturally to the GNS representation of the canonical bilateral extension of the process, in which we construct the canonical tail conditional expectation, Theorem \ref{condexpgen}, following \cite{LiuBM}, and we establish a non-commutative Olshen theorem, Theorem \ref{HSonZ}, identifying the tail algebra with the stationary algebra and, in the exchangeable case, with the exchangeable algebra. Note that the latter result is known not to hold in general in the GNS representation of the original process without faithfulness assumptions  \cite{DKW}.
We then prove that the canonical tail conditional expectation of a spreadable process satisfies both the conditional forward-decoupling factorization property and K\"{o}stler's full conditional independence property, Theorem \ref{deFinetti}. Besides extending the de Finetti Theorem in \cite{K} beyond the faithful setting, Theorem \ref{deFinetti} shows that the factorization singled out by the Hewitt-Savage ergodic level naturally persists in the conditional setting. 
This raises the natural question of whether adding the conditional forward-decoupling property restores the converse direction of K\"{o}stler's de Finetti theorem, namely whether stationarity together with full conditional independence and the conditional forward-decoupling property suffices to recover spreadability. The answer is again negative. Indeed, Gohm and K\"{o}stler's original counterexample in \cite{GK08} can be seen to satisfy the conditional forward-decoupling property as well, showing that this additional factorization is still insufficient to characterize spreadability.
A closer inspection of the proofs however reveals that the full strength of spreadability is only required to derive K\"{o}stler's full conditional independence. The remaining structural results rely instead on a weaker remnant of spreadability, which still suffices to establish the non-commutative Hewitt-Savage $0$-$1$ law, the associated hierarchy of ergodic properties, the construction of the canonical tail conditional expectation together with the Olshen theorem, and the conditional forward-decoupling factorization. In particular, we mention that the latter implies K\"{o}stler's ordered conditional independence. This observation naturally leads to the notion of weak spreadability, introduced in Definition \ref{weakspread}.

We show that weak spreadability is the minimal distributional symmetry underlying the non-commutative de Finetti theory developed in this paper. Besides yielding all the structural results described above, it is exactly characterized by the following de Finetti theorem, Theorem \ref{deFinetticonv}: a process is weakly spreadable if and only if it is stationary and admits a conditional expectation onto the tail algebra satisfying the forward-decoupling property.

Finally, we show in Corollary \ref{deFinettiquot} that weak spreadability is a genuinely non-commutative phenomenon. In the commutative setting, weak spreadability coincides with spreadability/exchangeability, and more generally the same remains true for processes whose degree of non-commutativity is dialed down in a suitable sense. In particular, our de Finetti theorem, Theorem \ref{deFinetticonv}, recovers the classical de Finetti theorem exactly.

To complete the picture, the present work, together with the existing non-commutative de Finetti theorems in the literature yields a unified ergodic characterization of Muraki's five natural notions of independence for infinite i.i.d. sequences. Each of them is precisely obtained by combining the Hewitt-Savage ergodic condition with the stronger distributional symmetry appearing in the corresponding de Finetti theorem (or through a quotient factorization property in the tensor case).\\

We now move on to describe how the sections of the paper are organized.
In Section \ref{prel}, we establish the framework of $C^*$-quantum probability spaces as a setup for quantum stochastic processes, which are here assumed to be unital for simplicity, with the general case of possibly non-unital processes being deferred to Section \ref{nunital}.
We discuss the canonical bilateral  extension procedure which allows us to obtain canonically, by stationarity, a double infinite quantum sequence of random variables out of a process indexed by the naturals. \\
In Section \ref{HS}, we provide the non-commutative formulation of the Olshen theorem, Theorem \ref{HSonZ}, that for a spreadable process the stationary algebra and the tail algebra in the Hilbert space of the bilateral extension are the same, and the tail algebra of the process is isomorphic with the tail algebra in the larger Hilbert space via the natural restriction $*$-automorphism.
We also prove a Hilbert-space version of the Olshen theorem, Theorem \ref{HSspreadonL2}, which settles invariant vectors rather than invariant operators.\\
In Section \ref{conexsec}, we tackle the problem of establishing the existence of a unique normal, invariant conditional expectation onto the tail algebra in the GNS representation of the bilateral extension of the given process for a (weakly) spreadable quantum stochastic process, Theorem \ref{condexpgen}. Among other things, the existence of such a conditional expectation leads us to  a characterization of extremality (of the distribution of the process) in terms of pureness of the restriction of the distribution to the tail algebra, Corollary \ref{purerestriction},
completing the partial characterization in \cite[Proposition 5.4.3]{DKW}.\\
In Section \ref{01law}, we state and prove in Theorem \ref{01lawspread} the general non-commutative Hewitt-Savage $0$--$1$ law.
Theorems \ref{ergohierexchange} and \ref{ergohierspread} cover the hierarchy of ergodicities in the exchangeable and spreadable case, respectively. Proposition \ref{tensor} and Remark \ref{free} contain the discussion of how identically distributed tensor or free independent processes can be characterized in terms of the suitable invariance of the distribution combined with any of the three conditions in Theorem \ref{01lawspread}.\\
Section \ref{definettisec} is devoted to the complete development of weak spreadability and the general non-commutative de Finetti Theorem.\\ 
In Section \ref{nunital}, we complete our theory by dealing with possibly non-unital processes. The resulting analysis, which up to some techicalities is the same as the unital case, is summed up in Theorem \ref{mainnonunital}, where the four levels  of ergodicities
for non-unital processes are spelled out. By enlarging the class of processes dealt with, we can also take into account Boolean and monotone/antimonotone independence. In particular, in Examples \ref{boolean} and \ref{monotone} we  discuss how, in the assumptions used in the literature \cite{Liu,LiuBM}, identically distributed Boolean or monotone/antimonotone independent processes can be characterized in terms of the suitable invariance of the distribution combined with any of the three conditions in Theorem \ref{01lawspread}.\\
Section \ref{models} is devoted to concrete examples and counterexamples illustrating the sharpness of the results established throughout the paper. We investigate classes of quantum processes arising from quotients of free products, including the infinite non-commutative torus. In particular, these examples show that the natural non-commutative analogue of moment factorization does not suffice to guarantee the triviality of the tail algebra. We also obtain a detailed description of the tail algebra of the extreme spreadable states on the infinite non-commutative torus, viewed as the distributions of quantum processes satisfying the twisted tensor commutation relations of the non-commutative torus. Finally, we observe that the techniques developed extend naturally to other twisted tensor products, including parafermion algebras \cite{BJLW}.
\section{Preliminaries}\label{prel}

 \subsection{Quantum Stochastic Processes and Distributional Symmetries}

We start by establishing  the non-commutative framework we work in, which is the $C^*$-algebraic approach.
A $C^*$-probability space is a pair $(\gb, \om)$  of a unital $C^*$-algebra $\gb$ and a state $\om: \gb\rightarrow\bc$.
A (unital) quantum stochastic process with (unital) sample  $C^*$- algebra $\ga$ is   a quadruple $(\gb, \om, \ga, \{\iota_j\}_{ j\in\bn})$,
where $(\gb, \om)$ is a $C^*$- probability space and $\{\iota_j: j\in\bn\}$  is a family of
unital $^*$-homomorphisms $\iota_j:\ga\rightarrow\gb$, which are to be interpreted as the non-commutative
random variables. 
\begin{rem}
For the time being, we will be focused on the unital case for sake of simplicity. However, the more general case
of possibly non-unital sample algebras and/or non-unital $*$-homomorphisms will be handled in Section \ref{nunital}.
A few words on the approach we follow are now in order. First, the definition of non-commutative random variable  includes the classical case. Indeed, the situation of a real-valued random variable corresponds exactly to taking $\ga=C_0(\br)$
and $\iota(f)=f(A)$, $f$ in $C_0(\br)$, the continuous functional calculus of some possibly unbounded self-adjoint operator $A$ on $\ch_\om$ (the GNS space of $\pi_\om$, the GNS representation of $\om$). The case of a whole  sequence of random variables is obtained
by considering $\ga=C_0(\br)$ and $\iota_j(f)=f(A_j)$, $j$ in $\bn$, for a sequence of possibly unbounded operators commuting with one another (that is for all $j, k$ in $\bn$ one has $f(A_j)g(A_k)=g(A_k)f(A_j)$ for all $f, g$ in $C_0(\br)$).\\
In particular, the approach based on possibly non-unital $*$-homomorphisms allows one to consider unbounded random variables
while sticking to a $C^*$-algebra formalism.
\end{rem}

The distribution of a quantum stochastic process as above is the state $\varphi$ on the infinite unital free product (with amalgamation of the unit)
$\ast_\bn\ga$ given by
$$\varphi (i_{j_1}(a_1)i_{j_2}(a_2)\cdots i_{j_n}(a_n)):=\om(\iota_{j_1}(a_1)\iota_{j_2}(a_2)\cdots\iota_{j_n}(a_n))$$
for all $n\in\bn$, $j_1\neq j_2\neq \cdots \neq j_n\in\bn$, where $i_k:\ga\rightarrow\ast_\bn\ga$ is the $k$-th embedding of $\ga$ into its infinite free product $\ast_\bn\ga$.
That the above formula defines a state on $\ast_\bn\ga$ can be seen as follows. 
Denote by $\pi_\om$ the GNS representation of $\cb$ corresponding to the state $\om$. 
Consider the family
$\{\pi_\om\circ\iota_j\}_{j\in\bn}$ of $*$-homomorphisms of $\ga$. By the universal property of the free product there exists a unique representation $\pi: \ast_\bn\ga\rightarrow \cb(\ch_\om)$ such that
$\pi(i_{j_1}(a_1)i_{j_2}(a_2)\cdots i_{j_n}(a_n))=\pi_\om(\iota_{j_1}(a_1)\iota_{j_2}(a_2)\cdots\iota_{j_n}(a_n))$, which means
$\varphi$ is nothing but the vector state associated with $\xi_\om$ of the representation $\pi$.\\
Conversely, one can obtain a quantum stochastic process from any given  state $\varphi$ on $\ast_\bn\ga$ with $\varphi$ itself as its distribution.
This will have $\ga$ as its sample algebra, $\cb=\cb(\ch_\varphi)$, $\om$ the vector state corresponding to $\xi_\varphi$, and finally
$\iota_j= \pi_\varphi\circ i_j$ for every $j$ in $\bn$.\\
A quantum stochastic process is identically distributed if, for all $j, k$ in $\bn$, $\om\circ\iota_j=\om\circ\iota_k$ (or equivalently
$\varphi\circ i_j=\varphi\circ i_k$) as states on the sample algebra $\ga$.\\
The tail algebra of a quantum stochastic process will play a key role throughout the paper.
In complete analogy with the classical case, it can be defined as the von Neumann algebra acting
on $\ch_\varphi$, the GNS space of $\pi_\varphi$, given by
\begin{equation}\label{tail}
\mathcal{T}_\varphi:=\bigcap_{n\in\bn}\{\pi_\varphi(i_k(\ga)): k\geq n\}''=\bigcap_{n\in\bn}\{\pi_\om(\iota_k(\ga)): k\geq n\}''\, .
\end{equation}
for a quantum stochastic process with distribution $\varphi$.\\
If we start directly with a state $\varphi$ on $\ast_\bn\ga$, we can still define its tail algebra $\mathcal{T}_\varphi$ by the first equality above, or equivalently as the tail algebra of the process induced by $\varphi$.

We now move on to recall the main distributional symmetries coming from the classical setting.  To do so, we need to
remind the reader what the natural actions of permutations/monotone maps by indices on the free infinite product look like.\\
Denote by $\bp_\bn$ the group of bijective maps of $\bn$
acting non-trivially 
only on  finitely many naturals. The
group operation is given by the composition.\\
The universal property of the free product C$^*$-algebra ensures 
 that any function
$f: \bn\to\bn$ 
can be lifted to a $*$-endomorphism 
$\Phi_f: *_{\bn}\ga\to *_{\bn}\ga$,
 uniquely determined by 
$\Phi_f(i_j(a))=i_{f(j)}(a)$ for all $j\in\bn$, $a\in \ga$. 
In particular, the group $\bp_\bn$ has a natural action
$\alpha$ of $\bp_\bn$ on $*_\bn\ga$ as a group homomorphism $\alpha : \bp_\bn \to \aut(*_\bn\ga)$.\\
Any state $\varphi$ in $\cs(*_\bn\ga)$
invariant under the action of $\alpha$, i.e. such that 
$\varphi\circ\alpha_\sigma=\varphi$ for all $\sigma\in\bp_\bn$, is called exchangeable, and so is the quantum stochastic process
with  distribution $\varphi$.\\
Stationary states are those invariant under the shift, 
namely the endomorphism $\tau: *_\bn\ga\to *_\bn\ga$,
which is defined as $\tau(i_j(a)):=i_{j+1}(a)$. Accordingly, a process whose distribution is stationary will be called stationary as well.\\
Now consider the semigroup of all strictly increasing maps $h: \bn\to\bn$.
Each of these maps induces an endomorphism of  $*_\bn\ga$.
Any state $\cs(*_\bn\ga)$
invariant under all these endomorphisms is called spreadable.
Basically, a process is spreadable if and only if all of its subsequences have the same joint distribution as the whole process.
For an analysis of the relations between spreadability and exchangeability in the non-commuttaive setting the reader is referred to
\cite{ADR}. Here we limit ourselves to recalling that spreadability is a weaker symmetry than exchangeability.
Instead of working with the semigroup of all strictly increasing monotone maps, we may as well work with the smaller
semigroup $\bj_\bn$ of those stricly increasing maps $h$ with cofinite range, $|\bn\setminus h(\bn)|<\infty$. Indeed, verifying spreadability amounts to checking invariance at the level of the dense $*$-algebra generated by finite words in $\ast_\bn\ga$, and on any finite set of $\bn$ the action of a stricly increasing map $h$ coincides with that of some $\widetilde{h}$ in $\bj_\bn$.\\
Recall that $\bj_\bn$ is generated as a semigroup by the partial shifts $\theta_h$, $h\in \bn$, 
\[
\theta_h(k) := 
\begin{cases} 
k & \text{if } k < h, \\ 
k + 1 & \text{if } k \geq 0\, ,
\end{cases}
\]
see \cite{K}.
In order to refer to the various sets of invariant states, we set up the notation
$\cs^{\bp_\bn}(\ast_\bn\ga)$, $\cs^{\bj_\bn}(\ast_\bn\ga)$, and $\cs^\tau(\ast_\bn\ga)$ to denote the sets of exchangeable,
spreadable, and stationary states, respectively.\\

In many  examples the variables of the quantum process may satisfy  relations, with the classical case corresponding to
(a commutative sample algebra and) commuting variables. For this reason, it is useful to establish terminology tailored to the
situations where the distribution of the process actually comes from a suitable quotient of the infinite free product, as is done in the following definition.

\begin{defin}\label{distri}
We say that the distribution $\varphi$ of a quantum stochastic process with sample algebra $\ga$
 factorizes through the quotient $\ast_\bn\ga/I$, for some two-sided ideal $I$,  if
$\varphi=\widetilde{\varphi}\circ p_I$ for some state $\widetilde{\varphi}$ on $\ast_\bn\ga/I$, with
$p_I: \ast_\bn\ga\rightarrow \ast_\bn\ga/ I$ being the canonical
projection. 
\end{defin}

\begin{rem}\label{censupport}
 Many results in the $W^*$-framework  used in the literatrure are stated under the assumption that the distribution is a faithful state on the von Neumann algebra generated by the process (see e.g. \cite{BCS,GK08,K,KS}).
 The corresponding hypothesis in the $C^*$-setting is that the GNS vector $\xi_\varphi$ be separating for the von Neumann algebra generated by the process, $\mathcal{R}_\varphi:=\pi_\varphi(\ast_\bn\ga)''$.\\
 It is worth recalling that, in general, in order for the GNS vector $\xi_\varphi$ to be separating for the von Neumann algebra generated by the GNS representation of the state $\varphi$ on a $C^*$-algebra
$\gb$, it is necessary and sufficient that the support of $\varphi$ in the bidual $\gb^{**}$ lies in the center of $\gb^{**}$, see e.g. \cite[p. 15]{NSZ}. Since the latter terminology is more concise, even though we will only make use of the former characterization, we shall simply say that the state has central support when referring to this property. Clearly, all of the states of a commutative $C^*$-algebra have central 
support.
 \end{rem}


%
 We end this preliminary present section with some notation used consistently in the proofs of many results of the paper.\\
If $\Phi$ is a $*$-endomorphism of a $C^*$-algebra $\gb$ and $\varphi$ is a $\Phi$-invariant state, then it is possible to define an isometry
$V_\Phi$ on $\ch_\varphi$ as $V_\phi \pi_\varphi(b)\xi_\varphi:=\pi_\varphi(\Phi(b))\xi_\varphi$, for all $b$ in $\gb$.
We mention that $V_\Phi$ is also known as the implementing isometry of $\Phi$ on the GNS representation of $\varphi$, although
$V_\Phi\pi_\varphi(b)V_\Phi^*$ may in general differ from $\pi_\varphi(\Phi(b))$. However, if $\Phi$ is a $*$-automorphism, in which case $V_\Phi$ is a unitary, then the adjoint action of $V_\phi$ does implement the action of $\Phi$.\\

\subsection{Bilateral Extension  and Implementation of the Shift}

Because the shift $\tau$ acts on the infinite free product indexed by $\bn$, $\ast_\bn \ga$, as  a proper $*$-endomorphism
($\tau$ is injective but not surjective), it need not lift to a normal $*$-endomorphism of the von Neumann algebra generated by the GNS representation of any given shift-invariant state. By normal implementation of the shift in the GNS representation of a stationary state $\varphi$ we mean the normal $*$-endomorphism $\widetilde{\tau}$ of $\pi_{\varphi}(\ast_\bn\ga)''$ uniquely determined by $\widetilde{\tau}(\pi_\varphi(x))=\pi_\varphi(\tau(x))$, $x$ in $\ast_\bn\ga$. When such an implementation exists, we simply denote it by $\tau$ since it is uniquely determined by the shift at the $C^*$-algebra level.\\
As shown in \cite{ DKW, LiuBM}, examples can be given of exchangeable states in whose GNS representation the shift cannot implemented as a normal $*$-endomorphism. In such cases, in the GNS representation of the given state, as a drawback of the fact that the shift does not act as a normal $*$-endomorphism on the whole von Neumann algebra,
it is not possible
to identify  the tail algebra with the shift-invariant von Neumann subalgebra,   nor is it possible to exhibit a normal conditional expectation $E$ onto the tail algebra which is shift-invariant  (that is $E(\pi_\varphi(x))=E(\pi_\varphi(\tau(x)))$ for all $x$ in $\ast_\bn\ga$). For completeness' sake we recall here what is meant by normal conditional expectation in the non-commutative setting an discuss an example of a stationary state for which the shift admits no normal implementation.
\begin{defin}\label{conditional}
A normal conditional expectation $E: \mathcal{M}\rightarrow\mathcal{N}$ from a von Neumann algebra $\mathcal{M}$ onto a 
 von Neumann subalgebra $\mathcal{N}\subseteq\mathcal{M}$ is a normal, contractive, completely  positive projection onto
$\mathcal{N}$ ($E=E^2$ and $E(n)=n$ for all $n$ in $\mathcal{N}$)  such that $E(nmn')=nE(m)n'$ for all $m$ in $\mathcal{M}$ and $n, n'$ in $\mathcal{N}$.
\end{defin}

\begin{example}\label{noshift}
In \cite[ Example 5.2.1]{DKW}  the authors consider the sample algebra algebra
$\ga=\bc^3$ with generator $a=(1, -1, 0)$. Consider the Hilbert space
$\ell^2(\bn_0)$ with canonical orthonormal basis  $\{v_n:n\geq 0\}$.
For every $n\geq 1$, define the representations 
$\iota_n:\ga\rightarrow\cb(\ell_2)$ by
$$\iota_n(a)= e_{0,n}+e_{n, 0}\,,$$ 
where $e_{n, m}$ is the rank-one operator
$e_{n, m}(x) =\langle x,  v_m\rangle v_n$, $x$ in $\ell^2(\bn_0)$.
For every finite permutation $\s$ of $\bn$, consider the unitary $U_s$ acting on $\ell^2(\bn)$ as
$U_\s v_n=v_{\s(n)}$ $U_\s v_0=v_0$.\\
Let $\pi: \ast_\bn\ga\rightarrow\cb(\ell^2(\bn_0))$ be the corresponding representation, that is the representation of $\ast_\bn\ga$ uniquely determined by
$\pi(i_k(a))=	\iota_k(a)$, $k\in\bn$.  Consider the state $\varphi$ on $\ast_\bn\ga$ given by $\varphi(x)= \langle \pi(x)v_0, v_0 \rangle$, $x$ in $\ast_\bn\ga$. As observed in the mentioned paper, the GNS representation
of $\varphi$ is $\pi$ itself since $v_0$ is cyclic ($\pi$ is actually irreducible in that its range coincides with the unitalization of the algebra of compact operators on $\ell^2(\bn)$).  We next show that $\tau$ does not extend
to a normal $*$-endomorphism of $\cb(\ell^2(\bn))$.  To this end, we start by noting that $e_{0, 0}=\iota_n(a^2)\iota_{n+1}(a^2)$ and
$e_{n, n}=\iota_n(a^2)-\iota_n(a^2)\iota_{n+1}(a^2)$, for all $n\geq 1$.
Therefore, if the shift could be extended,  we would 
  have $\tau(e_{0, 0})=e_{0, 0}$ and $\tau(e_{n, n})=e_{n+1, n+1}$ for all $n\geq 1$.
However, such a $\tau$ cannot be normal. Indeed,  from $\sum_{n\geq 0}{ e_{n, n}}=I$ in the strong operator topology,
we find
\begin{align*}
I=\tau(I)=\tau(\sum_{n\geq 0}{ e_{n, n}})=\sum_{n\geq 0} \tau(e_{n, n})= e_{0, 0}+\sum_{n\geq 1} e_{n+1, n+1}= I-e_{1, 1}\,,
\end{align*}
which is a contradiction.\\
As shown in \cite{DKW}, the tail algebra $\mathcal{T}_\varphi$ is isomorphic with $\bc^2$ as it is the algebra generated by $e_{0, 0}$, and coincides with the subalgebra fixed by the adjoint action of all $U_\s$. However, no normal shift-invariant conditional expectation
$E: \cb(\ell^2(\bn))\rightarrow\mathcal{T}_\varphi$ exists, that is there does not exist any normal conditional expectation $E$ onto the tail algebra such that $E(\pi_\varphi(a))=E(\pi_\varphi(\tau(a)))$.
Indeed, by  stationarity we would
have $E(e_{0, 0})=e_{0, 0}$ and $E(e_{n, n})=0$ for all $n\geq 1$ (because $E (e_{n, n})= E(e_{n+k, n+k})$ for all $k$ and $e_{n+k, n+k}$ tends to $0$ in the weak operator topology as $k$ diverges). But then
$E(I)= E(\sum_{n\geq 0} e_{n, n})=e_{0, 0}\neq I$.
\end{example}

Notice instead that on the infinite free product indexed by $\bz$, $\ast_\bz\ga$, the shift $\tau$, which is again uniquely determined by
$\tau(i_k(a))=i_{k+1}(a)$, $k$ in $\bz$, $a$ in $\ga$, acts as a $*$-automorphism.\\
If one starts with a shift-invariant state $\varphi$ on $\ast_\bn\ga$, then there exists a unique shift-invariant extension,
$\varphi_\bz$, to the larger $C^*$-algebra  $\ast_\bz\ga$ by virtue of the following proposition. The state
$\varphi_\bz$ can be interpreted as  the distribution of the extended bilateral process, which is obtained by symmetry out of the original process indexed by $\bn$.

\begin{lem}
The restriction map $\cs(\ast_\bz\ga)^{\tau}\ni\eta\mapsto \eta\upharpoonright_{\ast_\bn\ga}\in\cs(\ast_\bn\ga)^{\tau}$ is
an affine homeomorphism of compact convex sets.
\end{lem}
\begin{proof}

The restriction map is easily seen to be injective. As for
its surjectivity, starting from a $\tau$-invariant state $\eta_0$ on $\ast_\bn\ga$, one first defines a sequence of states $\{\eta_n: n\in\bn\}$, where $\eta_n$ is the state on $\ast_{A_n}\ga$, with  $A_n:=\{-n,\ldots, 0, 1, \ldots\}$, defined as $\eta_n:=\eta_0\circ\tau^{n}$.  Now the sought shift-invariant extension
of $\eta_0$ is the unique state $\eta$ on $\ast_\bz\ga$ such that
the restriction of $\eta$ to $\ast_{A_n}\ga$ is $\eta_n$ for every $n$.\\
\end{proof}

For the sequel, we will need to consider the GNS representation, $\pi_{\varphi_\bz}$, of the extended state $\varphi_\bz$ on $\ast_\bz \ga$ as a representation of the smaller free product, $\ast_\bn\ga$. We then define the two von Neumann algebras
$$\mathcal{R}_{\varphi}:= \pi_\varphi(\ast_\bn\ga)''\,,\quad\mathcal{R}_{\varphi_\bz}:=\pi_{\varphi_\bz}(\ast_\bn\ga)'' ,$$
and note that $\ch_\bn:=\overline{\pi_{\varphi_\bz}(\ast_\bn\ga)\xi_{\varphi_\bz}}\,$ is a $\ast_\bn\ga$-invariant subspace, which is, as a representation of $\ast_\bn\ga$, unitarily equivalent to the GNS representation of the state $\varphi$.\\
Finally, we denote by $\Phi: \mathcal{R}_{\varphi_\bz}\rightarrow \mathcal{R}_{\varphi}$ the $*$-homomorphism given by
$\Phi(T)= T\upharpoonright_{\ch_\bn}$, $T$ in $\mathcal{R}_{\varphi_\bz}$.

The following Proposition offers an initial indication that the GNS representation of the canonical bilateral extension $\pi_{\varphi_\bz}$ is the natural framework for our analysis.
 
\begin{prop}\label{restriction}
The restriction $*$-homomorphism $\Phi: \mathcal{R}_{\varphi_\bz}\rightarrow \mathcal{R}_{\varphi}$ is a $*$-isomorphism of von Neumann algebras
if and only if the shift $\tau$ of  $\ast_\bn\ga $ admits a (unique) normal implementation on $\mathcal{R}_\varphi$.
  \end{prop}

  \begin{proof}
  If $\Phi$ is a $*$-isomorphism, then the normal implementation of the shift on $\mathcal{R}_\varphi$ is obtained
  as $\Phi\circ\tau\circ\Phi^{-1}$, where $\tau$ is the shift on $\mathcal{R}_{\varphi_\bz}$.\\
  For the converse implication, let us suppose that the shift admits a normal implementation $\widetilde{\tau}$ on $\mathcal{R}_\varphi$. As the equality  $\Phi\circ {\rm ad} U_\tau=\widetilde{\tau}\circ \Phi$, where $U_\tau$ is the implementing unitary of $\tau$ (thought of as a $*$-automorphism of $\ast_\bz\ga$) in
$\pi_{\varphi_\bz}$,  clearly holds at the $C^*$-algebra level, it continues to hold at the von Neumann level as well. In particular,  $\rm{Ker}\,\Phi$ is a shift-invariant
ideal. We are now ready to show that $\Phi$ is injective. If $T$ in $\mathcal{R}_{\varphi_\bz}$ is such that $\Phi(T)=0$, then by shift-invariance of  $\rm{Ker}\,\Phi$ we must have $ \Phi(U_\tau^k T U_\tau ^{-k})=0$ for all naturals $k$, that is
 $ U_\tau^k T U_\tau ^{-k}\pi_{\varphi_\bz}(x)\xi_{\varphi_\bz}=0$ for all $x$ in $\ast_\bn\ga$, thus
 $ T \pi_{\varphi_\bz}(\tau ^{-k}(x))\xi_{\varphi_\bz}=0$ (the $\tau$ appearing in the right-hand side of this equality is the shift automorphism on $\ast_\bz\ga$). Since $\bigcup_k \tau ^{-k}(\ast_\bn\ga)$ is a norm dense $*$-algebra of 
$\ast_\bz\ga$, we finally see that $T$ vanishes on a norm dense subspace and thus $T=0$.
  \end{proof}

\begin{rem} We would like to observe that if $\varphi$ has central support, see Remark \ref{censupport}, then $\tau$ certainly admits a normal implementation on $\mathcal{R}_\varphi$, which means Proposition \ref{restriction} applies yielding the $*$-isomorphism $\mathcal{R}_{\varphi_\bz}\cong \mathcal{R}_{\varphi}$. In particular, this applies to the classical/commutative case.\\
To see that the shift has a normal implementation, note that
$\pi_\varphi(x)=0$  implies $\pi_\varphi(\tau(x))=0$ by the assumption on the central support of $\varphi$.
Therefore, the map $\pi_\varphi(x)\mapsto \pi_\varphi(\tau(x))$ is a well-defined $*$-endomorphism of $\pi_\varphi(\ast_\bn\ga)$.
Furthermore, since $\varphi$ is shift-invariant, the above map is isometric w.r.t. the GNS Hilbert space topology. The conclusion then follows from this general fact: the Hilbert space topology induced by a faithful state on a von Neumann algebra is the same as its strong operator topology on bounded sets.
\end{rem}

We end the section with the following terminology that will be used throughout the paper.
\begin{defin}\label{localized}
An element $x\in\ast_\bn\ga$ (resp. $x\in\ast_\bz\ga$, $x\in \mathcal{R}_{\varphi_\bz}$) is said to be \emph{localized} if $x\in \ast_F \ga \subset \ast_\bn\ga$ (resp. $x\in \ast_F \ga \subset \ast_\bz\ga$, $x\in  \pi_{\varphi_\bz}(\ast_F \ga)'' \subset \mathcal{R}_{\varphi_\bz}$), where $F\subset\bn$ (resp. $F\subset\bz$) is a finite subset and the inclusion is the obvious one. In such a case we will say that the support of $x$ is contained in $F$. Given two (finite) subsets $F, G\subset\bn$, we wite
$F< G$ to mean that $\max F < \min G$, and analogously for $F>G$.\\
For localized elements $x,y$ we will say that the support of $x$ lies to the left/right of (resp. is disjoint from) that of $y$ if $F_X<F_Y/ F_X>F_Y$ (resp. $F_X\cap F_Y=\emptyset$) for some finite sets $F_X, F_Y$ which contain the support of $x,y$ respectively.

\end{defin}

\section{Non-commutative Olshen theorem}\label{HS}

It is a known fact that the Olshen theorem may fail to hold in the GNS representation of an exchangeable state on
the free product indexed by $\bn$. One reason of the failure is that  the stationary algebra might not be defined, as a consequence of the fact that the shift itself may not be implemented as a normal $*$-endomorphism on the GNS representation of the given state, see Example \ref{noshift}.
Another is that the tail algebra may be  stricly included in the exchangeable algebra. Counterexamples of this sort show up even in the very tame case of processes whose distribution factorizes through tensor products in the sense of Definition \ref{distri}. In the following example, $p: \ast_\bn \ga \rightarrow \otimes_\bn \ga$ is the canonical quotient projection.

\begin{example}\label{tensprodex}
On the sample algebra $\ga:=\bm_2(\bc)$, we consider the state $\om(A)=A_{1,1}$ ($\om$ is the vector state associated with $e_1$, with
$\{e_1, e_2\}$ being the canonical orthonormal basis of $\bc^2$).\\
 On the infinite tensor product $\otimes_\bn\ga$, we consider the product state $\otimes\om$, which is obviously permutation
invariant.
Since $\om$ does not have central support, in the sense recalled in Remark \ref{censupport}, its product $\otimes\om$ does not have central support either, as follows from
{\it e.g.} Proposition \ref{csupport}.
Now, if the fixed-point algebra under permutation in the GNS representation of ($\otimes\om\circ p$, $\ast_\bn \ga$), 
$(\pi_{\otimes \omega \circ p}(\ast_\bn \ga)'')^{\bp_\bn}$,
 were  trivial, then
by \cite[Theorem 4.3.20]{BR1} the product state would have central support, which is not the case.\\
On the other hand, the tail algebra of $(\otimes\om \circ p)$ is trivial, as could be verified directly. However, the fact that the
tail algebra is trivial can also be seen as a consequence of Theorem \ref{01lawspread}.
\end{example}


 We complete the description of our setting with a few definitions and notation. We recall that
$\pi_{\varphi_\bz}$ denotes the GNS representation of the canonical extension $\varphi_\bz$ on $\ast_\bz \ga$, thought of as a representation
of the smaller free product, $\ast_\bn\ga$.\\
The tail algebra of $\varphi$ seen in the larger Hilbert space of the bilateral extension is
$$\mathcal{T}_{\varphi_\bz}:= \bigcap_n\{ \pi_{\varphi_\bz}(\iota_k(\ga)): k\geq n\}''\, .$$
The shift-invariant algebra is
$$\mathcal{R}_{\varphi_\bz}^{\tau}:= \{T\in \mathcal{R}_{\varphi_\bz}: TU_\tau=U_\tau T \}\,.$$
Finally, if the state $\varphi$ is also exchangeable, then we can define the exchangeable algebra as well by
$$\mathcal{R}_{\varphi_\bz}^{\bp_\bz}:= \{T\in \mathcal{R}_{\varphi_\bz}: TU_\s=U_\s T, \s\in\bp_\bz \}\,,$$
where $U_\s$ is the unitary operator acting on $\ch_{\varphi_\bz}$ that implements the automorphism of
$\ast_\bz\ga$ induced by the permutation $\s$ in $\bp_\bz$.\\

We are ready to state the non-commutative Olshen Theorem.

\begin{thm}\label{HSonZ}
For any spreadable state $\varphi$ on $\ast_\bn\ga$ one has 
$$\mathcal{R}_{\varphi_{\bz}}^{\tau}=\mathcal{T}_{\varphi_{\bz }}\, .$$
Moreover,  the restriction map $\Phi: \mathcal{T}_{\varphi_\bz}\rightarrow \mathcal{T}_\varphi$ is a $*$-isomorphism.\\
If $\varphi$ is also exchangeable, then  one has
$$\mathcal{R}_{\varphi_\bz}^{\bp_\bz}=\mathcal{R}_{\varphi_\bz}^{\tau}=\mathcal{T}_{\varphi_\bz}\,.$$
\end{thm}

\begin{proof}
We start by dealing with the exchangeable case, where more equalities need to be proved.\\
The inclusions $\mathcal{R}_{\varphi_\bz}^{\tau}\subseteq\mathcal{T}_{\varphi_\bz}\subseteq\mathcal{R}_{\varphi_\bz}^{\bp_\bz}$ can be verified as follows. For the first, any shift-invariant $T$ lies {\it a fortiori} in the tail algebra $\mathcal{T}_{\varphi_\bz}$ due to the equality
$T=\tau^n(T)$, for all $n$. For the second, any $T$ in the tail algebra $\mathcal{T}_{\varphi_\bz}$ is fixed by all permutations: if $\s$ is a permutation such that $\sigma(k)=k$ for all $k\geq n$ then $U_\s TU_\s ^*=T$  as $T$ lies in $\{ \pi_{\varphi_\bz}(\iota_k(\ga)): k\geq n\}''$.\\
The chain of equalities in the statement will then follow as long as we show that 
$\mathcal{R}_{\varphi_\bz}^{\bp_\bz}\subseteq\mathcal{R}_{\varphi_\bz}^{\tau}$.
To this end, we claim that the shift can be obtained as the following limit $\tau=\lim_n {\rm ad}\, (U_{\s_{-n}}\cdots U_{\s_n})$, which holds pointwise in the strong operator topology ($\s_i$ is the permutation that swaps $i$ and $i+ 1$). If $T$ in $\mathcal{R}_{\varphi_\bz}$ is fixed by all permutations, we have ${\rm ad}\, U_{\s_{-n}\cdots \s_n}(T)=T$ for all $n$, hence, taking the limit as $n$ goes to infinity, we find $\tau(T)=T$, {\it i.e.}
$T$ in $\mathcal{R}_{\varphi_\bz}^{\tau}$.\\
As for the claim, it is enough to show that the sequence of unitaries $U_{\s_{-n}}\cdots U_{\s_n}$ strongly converges to
$U_\tau$. Since the sequence $\{U_{\s_{-n}}\cdots U_{\s_n}: n\}$  is obviously bounded, it suffices to verify the limit equality on vectors of the form $\pi_{\varphi_\bz}(x)\xi_{\varphi_\bz}$ with $x=\iota_{i_1}(a_1)\cdots\iota_{i_k}(a_k)$ a localized monomial of the free product 
$\ast_\bz\ga$. But for such a vector we have $U_{\s_{-n}}\cdots U_{\s_n}\pi_{\varphi_\bz}(x)\xi_{\varphi_\bz}=\pi_{\varphi_\bz}(\tau(x))\xi_{\varphi_\bz}$ as soon as $n$ is greater than $\max\{|i_1|, \ldots,| i_k| \}$.\\
We now tackle the case of  spreadable state $\varphi$. Since the inclusion
$\mathcal{R}_{\varphi_{\bz}}^{\tau}\subseteq\mathcal{T}_{\varphi_{\bz }}$ is trivially satisfied, we need only show that
$\mathcal{T}_{\varphi_{\bz }}\subseteq\mathcal{R}_{\varphi_{\bz}}^{\tau}$ holds as well.
Let $T$ be an element of $\mathcal{T}_{\varphi_{\bz }}$. We will prove that $\tau(T)=T$ by showing that for any localized
$a, b$ in $\ast_\bz\ga$ one has
$$\langle \tau(T)\pi_{\varphi_\bz}(a)\xi_{\varphi_\bz}, \pi_{\varphi_\bz}(b)\xi_{\varphi_\bz} \rangle=\langle T\pi_{\varphi_\bz}(a)\xi_{\varphi_\bz}, \pi_{\varphi_\bz}(b)\xi_{\varphi_\bz} \rangle\, .$$
This can be seen as follows. Since $a, b$ are localized, there exists $n$ in $\bn$ such that the support of both  $a$ and $b$ lies to the left of $n$. Because $T$ sits in the tail algebra, we can assume its support lies entirely to the right of $n$. By spreadability of $\varphi$ we must have $\langle \tau(T)\pi_{\varphi_\bz}(a)\xi_{\varphi_\bz}, \pi_{\varphi_\bz}(b)\xi_{\varphi_\bz} \rangle=\langle T\pi_{\varphi_\bz}(a)\xi_{\varphi_\bz}, \pi_{\varphi_\bz}(b)\xi_{\varphi_\bz} \rangle$ by considering a monotone map of $\bz$ to itself which shifts by $+1$ the support of $T$ while not moving at all the supports of $a$ and $b$.\\
Finally, we are left with the task of proving that the restriction map $\Phi$ in the statement
is a $*$-isomorphism. To this end, we start by observing that the restriction map sends $\mathcal{T}_{\varphi_\bz}$ to  $\mathcal{T}_\varphi$ because  one has $\Phi(\{ \pi_{\varphi_\bz}(\iota_k(\ga)): k\geq n\}'')=\{ \pi_{\varphi}(\iota_k(\ga)): k\geq n\}''$ for every $n$.\\
 We move on to prove that the restriction map $\mathcal{T}_{\varphi_\bz}\ni T\mapsto T\upharpoonright_{\ch_\bn}\in\mathcal{T}_\varphi$  a $*$-isomorphism. We first show that $\Phi$ is surjective.
As already remarked,  for every $n$ one has $\Phi(\{ \pi_{\varphi_\bz}(\iota_k(\ga)): k\geq n\}'')=\{ \pi_{\varphi}(\iota_k(\ga)): k\geq n\}''$.  Now if $S$ lies in  $\bigcap_n\{ \pi_{\varphi}(\iota_k(\ga)): k\geq n\}''$, then there exist $T_n$ in 
$\{ \pi_{\varphi_\bz}(\iota_k(\ga)): k\geq n\}''$ such that $\Phi(T_n)=S$ and $\|T_n\|\leq \|S\|$ for all $n$.
By boundedness the sequence $\{T_n: n\in\bn\}$ has at least one accumulation point, say $T$, in the weak operator topology. By construction $T$ sits in $\bigcap_n\{ \pi_{\varphi_\bz}(\iota_k(\ga)): k\geq n\}''$ and $\Phi(T)=S$ by continuity of $\Phi$ w.r.t. the weak operator topology.\\
To conclude, let us prove that $\Phi$ is injective as well.
Suppose that $T$ in $\mathcal{T}_{\varphi_\bz}$ is such that $\Phi(T)=0$, that is
$T \pi_{\varphi_\bz}(x)\xi_{\varphi_\bz}=0$ for all $x$ in $\ast_\bn\ga$.
Thanks to the first part, such a $T$ is fixed by the shift, meaning we have
$(U_\tau)^nT(U_\tau)^{-n} \pi_{\varphi_\bz}(x)\xi_{\varphi_\bz}=0$, hence $T(U_\tau)^{-n} \pi_{\varphi_\bz}(x)\xi_{\varphi_\bz}=0$, from which it follows that $T=0$ as $\bigcup_n \tau ^{-n}(\ast_\bn\ga)$ is a norm dense $*$-algebra of $\ast_\bz\ga$.
\end{proof}
In order to state the Hilbert-space version of the non-commutative Olshen theorem, we first need to set some notation. In particular, we need to consider subpaces of fixed vectors under the action of the implementing isometries. If $\varphi$ is a spreadable
state of $\ast_\bn\ga$, then one may consider the subspace of shift-invariant vectors
$$\ch_{\varphi}^{\tau}:=\{x\in\ch_{\varphi}: V_\tau x=x\},$$
where $V_\tau$ is the proper isometry acting on $\ch_\varphi$ determined by $V_\tau \pi_\varphi(x)\xi_\varphi= \pi_\varphi(\tau(x))\xi_\varphi$, $x$ in $\ast_\bn\ga$.
We denote by $E_\varphi$ the orthogonal projection onto $\ch_{\varphi}^{\tau}$. We consider
the subspace of $\bj_\bn$-invariant vectors
$$\ch_{\varphi}^{\bj_\bn}:=\{x\in\ch_\varphi: V_hx=x\,\,\textrm{for all}\, h\in \bj_\bn\}\, ,$$
where $V_h$ is the isometry implementing $\a_h$, for all $h$ in $\bj_\bn$.\\
If the state $\varphi$ is also exchangeable, then we can define the subspace of exchangeable vectors as well, that is
$$\ch_{\varphi}^{\bp_\bn}:=\{x\in\ch_{\varphi}: U_gx=x\,\textrm{for all}\, g\in\bp_\bn  \}\, .$$

\begin{thm}\label{HSspreadonL2}
For any   spreadable state $\varphi$  on $\ast_\bn\ga$ one has
$$\ch_{\varphi}^{\tau}=\ch_{\varphi}^{\bj_\bn} = \overline{\mathcal{T}_\varphi \xi_\varphi}\, .$$
If $\varphi$ is also exchangeable, then
$$\ch_{\varphi}^{\tau}=\ch_{\varphi}^{\bj_\bn}=\ch_{\varphi}^{\bp_\bn}=\overline{\mathcal{T}_\varphi \xi_\varphi}.$$
\end{thm}
\begin{proof}
We start by dealing with the case of a spreadable state $\varphi$.\\
As for the inclusion $\ch_{\varphi}^{\tau}\subset \ch_{\varphi}^{\bj_\bn}$, if 
$V_Ni$ denote the isometry associated with the $N$-th partial shift $\theta_N$, we have the 
equality $V_N V_{N-1}=(V_{N-1})^2$, from which the inclusion follows as 
$V_0=V_\tau$.
The converse inclusion $\ch_{\varphi}^{\bj_\bn}\subset\ch_{\varphi}^{\tau}$ is trivial because the shift is one of the maps of $\bj_\bn$.\\
To conclude, we need to show
that $\ch_{\varphi}^{\tau}= \overline{\mathcal{T}_\varphi \xi_\varphi}$.\\
 The inclusion 
$ \overline{\mathcal{T}_\varphi \xi_\varphi}\subset \ch_{\varphi}^{\tau}$
will follow once we have proved that $\mathcal{T}_\varphi \xi_\varphi\subset \ch_{\varphi}^{\tau}$, because $\ch_{\varphi}^{\tau}$ is a closed subspace. If $T$ is in $\mathcal{T}_\varphi$, then by Theorem \ref{HSonZ} there exists a (unique) 
  $\widetilde{T}$ in $\mathcal{R}_{\varphi_\bz}^\tau$ such that $\Phi(\widetilde{T})=T$. Calculating the equality $U_\tau \widetilde{T}=\widetilde{T}U_\tau$ on $\xi_{\varphi_{\bz }}$ we find $V_\tau T\xi_\varphi=T\xi_\varphi$, {\it i.e.}
$T\xi_\varphi$ is in $ \ch_{\varphi}^{\tau}$.\\
Let us move on to prove the converse inclusion.
Let $y$ be a vector in $\ch_\varphi$ fixed by $V_\tau$. By cyclicity of $\xi_\varphi$, we can
obtain $y$ as a norm limit $y=\lim_k \pi_\varphi(x_k)\xi_\varphi$, for a suitable sequence
$\{x_k:k\in\bn\}$ of  localized elements of $\ast_\bn\ga$. 
Denote by $E_\varphi$ the orthogonal projection onto the fixed-point subspace of $V_\tau$.
We clearly have $y= E_\varphi y=\lim_k E_\varphi(\pi_\varphi(x_k)\xi_\varphi)$.\\
We now need to recall that by   von Neumann's ergodic theorem $E_\varphi=\lim_n \frac{1}{n}\sum_{k=0}^{n-1}V_\tau^k$, with the limit being understood in the strong operator topology.  But then for any $x$ in $\ast_\bn\ga$ we have:
\begin{align*}
E_\varphi(\pi_\varphi(x)\xi_\varphi)&=\lim_n \frac{1}{n}\sum_{k=0}^{n-1} V_\tau^k\pi_\varphi(x)\xi_\varphi=\lim_n \frac{1}{n}\sum_{k=0}^{n-1} (\pi_{\varphi_\bz}(\tau^k(x)))\xi_\varphi\\
&=\widetilde{T}\xi_\varphi\,
\end{align*}
where $\widetilde{T}$ is any accumulation point in the weak operator topology of the sequence $\{\frac{1}{n}\sum_{k=0}^{n-1} (\pi_{\varphi_\bz}(\tau^k(x))): n\in\bn\}$, which by construction is in  $\mathcal{R}_{\varphi_\bz}^\tau$. In particular, its restriction to $\ch_\varphi$ provides an element $T$ of the tail algebra such that $E_\varphi(\pi_\varphi(x)\xi_\varphi)= T\xi_\varphi$.\\
Finally, if $\varphi$ is also exchangeable, then $\ch_\varphi^{\bp_\bn}\subseteq\ch_\varphi^\tau$ due to
 the limit equality
$V_\tau=\lim_n U_{\s_{0}}\cdots U_{\s_n}$ (where $\sigma_i$ is the permutation that swaps $i$ and $i+1$), which can be verified as in the proof of Theorem \ref{HSonZ}. The converse inclusion $\ch_\varphi^\tau\subseteq\ch_\varphi^{\bp_\bn}$ follows from the equality
$\ch_\varphi^\tau=\overline{\mathcal{T}_\varphi \xi_\varphi}$ because the tail algebra $\mathcal{T}_\varphi$ is fixed by all permutations.
\end{proof}
For any spreadable state $\varphi$, we will henceforth denote by $E_\varphi$ the orthogonal projection onto any of the the subspaces above,
without any confusion occurring, in that they are all the same subspace by virtue of Theorem \ref{HSspreadonL2}.\\
As a first consequence of the above $L^2$ version of the Olshen Theorem, we show that the exchangeable states make up a face of the simplex of spreadable states, which is a face of the simplex of stationary states.

\begin{cor}\label{faces}
Each convex subset in the chain of inclusions $\cs^{\bp_\bn}(\ast_\bn\ga)\subset \cs^{\bj_\bn}(\ast_\bn\ga)\subset \cs^\tau(\ast_\bn\ga)$ is
a face of the convex sets in which it is contained.
\end{cor}

\begin{proof}
We will only prove that $\cs^{\bp_\bn}(\ast_\bn\ga)$ is a face of $\cs^\tau(\ast_\bn\ga)$. for the other cases can be handled in the exact same way.\\
Let $\varphi$ be an exchangeable state that is obtained as a non-trivial convex combination
$\varphi=t\om+(1-t)\eta$, $t\in (0, 1)$, of two stationary states $\om, \eta$. In particular, $\om$ is dominated by $\varphi$, and therefore 
there exists a positive operator $S$ in  $\pi_\varphi(\ast_\bn\ga)'$ such that
$\om(x)=\langle \pi_\varphi(x)S\xi_\varphi, \xi_\varphi\rangle$ for all $x$ in $\ast_\bn\ga$.\\
We claim that $S\xi_\varphi$ is in $\ch_{\varphi}^\tau$. If this is the case, then by Theorem \ref{HSspreadonL2} 
$S\xi_\varphi$ lies in $\ch_{\varphi}^{\bp_\bn}$ as well, meaning that $\om$ (and $\eta$) is exchangeable.\\
All we have to do to conclude is prove the claim.
For all $x$ in $\ast_\bn\ga$ we have
\begin{align*}
\langle \pi_\varphi(x)\xi_\varphi, S\xi_\varphi \rangle&=\om(x)=\om(\tau(x))=\langle \pi_\varphi(\tau(x))\xi_\varphi, S\xi_\varphi \rangle\\
&=\langle V_\tau\pi_\varphi(x)\xi_\varphi, S\xi_\varphi \rangle=\langle \pi_\varphi(x)\xi_\varphi, V_h^*S\xi_\varphi \rangle\,,
\end{align*}
which, by cyclicity of the GNS  vector implies $V_h^*S\xi_\varphi=S\xi_\varphi$. The conclusion now follows from this general fact: for any isometry $V$ on a Hilbert space $V^*y=y$ implies $Vy=y$.
\end{proof}

\section{Existence of the conditional expectation onto the tail algebra}\label{conexsec}

Examples of exchangeable states $\varphi$ on $\ast_\bn\ga$ are known for which  there is no invariant normal conditional expectation
from $\pi_\varphi(\ast_\bn\ga)''$ onto the tail algebra. For instance, such a situation occurs in the setting of Example \ref{noshift}, 
as a consequence of the fact that the shift cannot be implemented as a normal $*$-endomorphism of the von Neumann algebra generated in the GNS representation of the state. However, this difficulty can be overcome by
working in the GNS representation of the extended state $\varphi_\bz$, where the shift is implemented as the adjoint action of a unitary.
Importantly, the conditional expectation obtained from the shift is natural because the tail algebra defined on the larger Hilbert space
is $*$-isomorphic to the tail algebra in the $*$-isomorphism induced by the restriction map by Theorem \ref{HSonZ}.\\
Conditional expectations of this type were first constructed by K\"{o}stler in \cite{K} under the additional assumption that the invariant state is faithful  on the von Neumann algebra generated by the variables  ({\it i.e.} in our setting  the state has central support). 
Liu later removed the faithfulness assumption  in the $W^*$-setting  for processes indexed by $\bz$ \cite{LiuBM}.
Because our setting is slightly different, and the conditional expectation is crucial in our analysis, we include a complete self-contained proof of its existence.

\begin{thm}\label{condexpgen}
For any spreadable state  $\varphi$  on $\ast_\bn\ga$ there exists
a unique normal shift-invariant
conditional expectation $E_\tau:\mathcal{R}_{\varphi_{\bz }}\rightarrow\mathcal{T}_{\varphi_{\bz }}$, which is also
$\varphi$-invariant, {\it i.e.} $\langle E_\tau(T)\xi_{\varphi_\bz}, \xi_{\varphi_\bz} \rangle= \langle T\xi_{\varphi_\bz}, \xi_{\varphi_\bz} \rangle$
for all $T$ in $\mathcal{R}_{\varphi_{\bz }}$.\\
Moreover, for $T= \pi_{\varphi_{\bz }}(a)$, with $a$ being a localized element of $\ast_\bn\ga$, one has 
$$E_\tau(T)=\lim_{n\rightarrow\infty}\pi_{\varphi_{\bn }}(\tau^n(a))$$ in the weak operator topology of 
$\cb(\ch_{\varphi_{\bz }})$.
 $E_\tau$ is $\bj_\bn$-invariant, that is  for all $h$ in $\bj_\bn$ one has $E_\tau(\pi_{\varphi_\bz}(\a_h(a)))= E_\tau(\pi_{\varphi_\bz}(a))$,
$a$ in $\ast_\bn\ga$.\\
In addition, if $\varphi$ is also exchangeable $E_\tau$ is permutation invariant. 
\end{thm}

\begin{proof}
We start by establishing the existence of the stated limit
for $T= \pi_{\varphi_\bz}(a)$, for some localized $a$ in $\ast_\bn\ga$.
If  $b, c$ are localized elements in $\ast_\bz\ga$, one has
\begin{align*}
\langle \pi_{\varphi_\bz}(\tau^n (a)) \pi_{\varphi_\bz}(b)\xi_{\varphi_\bz}, \pi_{\varphi_\bz}(c)\xi_{\varphi_\bz}\rangle 
=&\varphi_\bz(c^*\tau^n(a)b)=\varphi_\bz(c^*\tau^{n+1}(a)b)\\
=&\langle \pi_{\varphi_\bz}(\tau^{n+1} (a)) \pi_{\varphi_\bz}(b)\xi_{\varphi_\bz}, \pi_{\varphi_\bz}(c)\xi_{\varphi_\bz}\rangle 
\end{align*}
for all $n$ sufficiently  large. Indeed, as soon as $n$ is so large that the support of $\tau^n(a)$ lies entirely to the right of the supports of both $b$ and $c$, there exists a monotone map that does not displace the supports of $b$ and $c$ and shifts by $1$ that of $\tau^n(a)$.
This shows that for any $x, y$ in the dense subspace $\pi_{\varphi_\bz}((\ast_\bz\ga)_{\rm {alg}})\xi_{\varphi_\bz}$, the limit
$\lim_n \langle \pi_{\varphi_\bz}(\tau^n (a))x, y\rangle $ exists and defines a bounded sequilinear form (whose norm is bounded by $\|a\|$). By boundedness, this form can be (uniquely) extended to a sequilinear form on the whole $\ch_{\varphi_\bz}$. By the Riesz representation theorem, this  form is associated with a unique bounded operator,
$E_\tau(\pi_{\varphi_\bz}(a))$, which by construction is the limit in the weak operator topology of the sequence
$\{\pi_{\varphi_\bz}(\tau^n (a)): n\in\bn\}$. By linearity the map $E_\tau$ is thus defined on the subalgebra $\pi_{\varphi_\bz}((\ast_\bz\ga)_{\rm {alg}})$.\\
We next note that for all localized $a$ in $\ast_\bn\ga$ and $b, c$ in $\ast_\bz\ga$ the equality
\begin{equation}\label{condexchange}
 \langle E_\tau(\pi_{\varphi_\bz}(a)) \pi_{\varphi_\bz}(b)\xi_{\varphi_\bz}, \pi_{\varphi_\bz}(c)\xi_{\varphi_\bz}\rangle
=\langle \pi_{\varphi_\bz}(a) \pi_{\varphi_\bz}(\tau^k(b))\xi_{\varphi_\bz}, \pi_{\varphi_\bz}(\tau^k(c))\xi_{\varphi_\bz}\rangle
\end{equation}
holds for every $k<0$ such that the supports of $\tau^k(b)$ and $\tau^k(c)$ lies to the left of $0$.\\
Indeed, since $E_\tau\circ\tau=\tau\circ E_\tau$, for all integers $k$ we have
$$\langle E_\tau(\pi_{\varphi_\bz}(a)) \pi_{\varphi_\bz}(b)\xi_{\varphi_\bz}, \pi_{\varphi_\bz}(c)\xi_{\varphi_\bz}\rangle=
\langle E_\tau(\pi_{\varphi_\bz}(a)) \pi_{\varphi_\bz}(\tau^{k}(b))\xi_{\varphi_\bz}, \pi_{\varphi_\bz}(\tau^k(c))\xi_{\varphi_\bz}\rangle$$
and the right-hand side ot the above equality can be expressed as
\begin{align*}
\lim_n \langle \pi_{\varphi_\bz}(\tau^n(a)) \pi_{\varphi_\bz}(\tau^{k}(b))\xi_{\varphi_\bz}, \pi_{\varphi_\bz}(\tau^k(c))\xi_{\varphi_\bz}\rangle=
 \langle \pi_{\varphi_\bz}(a) \pi_{\varphi_\bz}(\tau^{k}(b))\xi_{\varphi_\bz}, \pi_{\varphi_\bz}(\tau^k(c))\xi_{\varphi_\bz}\rangle
\end{align*}
thanks to $\bj_\bz$-invariance of $\varphi_\bz$ (for any fixed $n$, consider a permutation that fixes negative integers and shifts by $n$ the support of $a$).\\
If now $\{T_\a: \a \in I\}$ is  a bounded net of localized elements of $\pi_{\varphi_\bz}(\ast_\bz\ga)$ that converges to $0$ in the weak operator topology, we have that
\begin{equation}\label{wellposedness}
\lim_\a E_\tau(T_\a)=0
\end{equation}
 as well. 
Indeed, for all localized $b, c$ in $\ast_\bz\ga$, we have
\begin{align*}
&\lim_\a \langle  E_\tau(T_\a)\pi_{\varphi_\bz}(b)\xi_{\varphi_\bz}, \pi_{\varphi_\bz}(c)\xi_{\varphi_\bz}\rangle =\lim_\a \langle  T_\a (\pi_{\varphi_\bz}(\tau^k(b))\xi_{\varphi_\bz}, (\pi_{\varphi_\bz}(\tau^k(c))\xi_{\varphi_\bz}\rangle\\
&=0\,,
\end{align*}
for a suitable $k<0$, which gives the desired convergence by boundedness of the net $\{T_\a: \a \in I\}$.\\
We are now in a position to prove that the map $E_\tau$ can be extended to the whole von Neumann algebra $\mathcal{R}_{\varphi_\bz}$.
Let $T$ be an element of  $\mathcal{R}_{\varphi_\bz}$ and let $\{T_\a: \a \in I\}$ be a bounded net converging to $T$ in the weak operator topology.
A straightforward application of \eqref{condexchange} gives that $\{ E_\tau(T_\a): \a \in I\}$ is a (bounded) Cauchy net in the weak operator topology, and therefore it converges o a bounded
operator, which we may safely denote by $ E_\tau(T)$, in that it does not depend on the choice of the approximating net by \eqref{wellposedness}.\\
Equation \eqref{condexchange} continues to hold for $ E_\tau(T)$ for $T$ in the von Neumann algebra. Indeed, for all localized $b, c$ in $\ast_\bz\ga$, we have
\begin{align*}
&\langle E_\tau(T) \pi_{\varphi_\bz}(b)\xi_{\varphi_\bz}, \pi_{\varphi_\bz}(c)\xi_{\varphi_\bz}\rangle 
=\lim_\a \langle E_\tau(T_\a) \pi_{\varphi_\bz}(b)\xi_{\varphi_\bz}, \pi_{\varphi_\bz}(c)\xi_{\varphi_\bz}\rangle=\\
&\lim_\a \langle T_\a \pi_{\varphi_\bz}(\tau^k(b))\xi_{\varphi_\bz}, (\pi_{\varphi_\bz}(\tau^k(c))\xi_{\varphi_\bz}\rangle=\langle T \pi_{\varphi_\bz}(\tau^k(b))\xi_{\varphi_\bz}, (\pi_{\varphi_\bz}(\tau^k(c))\xi_{\varphi_\bz}\rangle\, .
\end{align*}
From this it easily follows that the restriction of the map $E_\tau$ to the unit ball of  $\mathcal{R}_{\varphi_\bz}$ is continuous in the strong/weak operator topology.\\
Note that the map $E_\tau$ is shift invariant, {\it i.e. } $E_\tau\circ\tau=\tau\circ E_\tau=E_\tau$.
This is in fact a straightforward consequence of the weak continuiy of $E_\tau$ and the fact that the equality
 holds on the dense $*$-subalgebra spanned by elements of the form $\pi_{\varphi_\bz}(a)$, with
$a$ in $\ast_\bn\ga$ localized, as we have already remarked. In particular, it follows that the range of $E_\tau$ is
contained in $\mathcal{T}_{\varphi_\bz}$.
We next show that the range of $E_\tau$ is the whole $\mathcal{T}_{\varphi_\bz}^\tau$ (that is $\mathcal{R}_{\varphi_\bz}^\tau$ by
Theorem \ref{HSonZ})
as a consequence of the fact that $T$ in $\mathcal{R}_{\varphi_\bz}^\tau$
implies $E_\tau(T)=T$.\\
Let $T$ be a shift-invariant element of $\mathcal{R}_{\varphi_\bz}$
and let $\{t_\a:\a \in I\}$ a  net in the dense subalgebra of localized elements of $\ast_\bn\ga$ such that
$\lim_\a \pi_{\varphi_\bz}(t_\a)=T$ in the weak operator topology. Suppose
$b, c$ are localized elements in $\ast_\bz\ga$ with supports contained in
$\{k\in\bz: k\leq N\}$. Since $T$ is shift-invariant, we may safely assume
that the support of each $t_\a$ lies to the right of $N$ (for if this were not the case, we could consider
the shifted net $\{\tau^N(t_\a):\a \in I\}$, which would still approximate $T$). But then we have

\begin{align*}
&\langle E_\tau(T) \pi_{\varphi_\bz}(b)\xi_{\varphi_\bz}, \pi_{\varphi_\bz}(c)\xi_{\varphi_\bz}\rangle =
\lim_\a \langle E_\tau(\pi_{\varphi_\bz}(t_\a)) \pi_{\varphi_\bz}(b)\xi_{\varphi_\bz}, \pi_{\varphi_\bz}(c)\xi_{\varphi_\bz}\rangle=\\
&\lim_\a \lim_n \langle (\pi_{\varphi_\bz}(\tau^n(t_\a)) \pi_{\varphi_\bz}(b)\xi_{\varphi_\bz}, \pi_{\varphi_\bz}(c)\xi_{\varphi_\bz}\rangle=\lim_\a\lim_n \varphi_\bz(c^*\tau^n(t_\a)b)\, .
\end{align*}
Now for any fixed $n$ and $\a$ there exists a monotone map that does not move the supports of $b, c$ and shifts by $n$ that of $t_\alpha$. By $\bj_\bz$-invariance of $\varphi_\bz$, we then have
$\varphi_\bz(c^*\tau^n(t_\a)b)=\varphi_\bz(c^*t_\a b)$. Going back to the previous computation, we find
\begin{align*}
&\langle E_\tau(T) \pi_{\varphi_\bz}(b)\xi_{\varphi_\bz}, \pi_{\varphi_\bz}(c)\xi_{\varphi_\bz}\rangle =
\lim_\a \varphi_\bz(c^*t_\a b)=\langle T \pi_{\varphi_\bz}(b)\xi_{\varphi_\bz}, \pi_{\varphi_\bz}(c)\xi_{\varphi_\bz}\rangle\, ,
 \end{align*}
hence $E_\tau(T)=T$ by cyclicity of the GNS vector.\\
In order to conclude that $E_\tau$ is  a conditional expectation, it is enough to make sure that 
$E_\tau$ is a contraction by Tomiyama's lemma see \cite[Theorem 1.5.10]{BO}. Note that $E_\tau$ is certainly contractive on
the dense subalgebra $\pi_{\varphi_\bz}((\ast_\bn\ga)_{\rm alg})$ because there $E_\tau$ is simply given as the limit of a sequence of $*$-endomorphisms. The general case is achieved by density and lower
semicontinuity of the operator norm, for if $\{T_\a: \a\in I\}$ is an approximating net of $T$ in
$\mathcal{R}_{\varphi_\bz}^\tau$  with $\|T_\a\|\leq \|T\|$ for all $\a$ (such a net exists by virtue of the Kaplansky density theorem) we have
\begin{align*}
\|E_\tau(T)\|\leq\limsup_\a \|E_\tau(T_\a)\|\leq \limsup_\a \|T_\a\|\leq \|T\|\, .
\end{align*}

As for the uniqueness of $ E_\tau$, suppose $F$ is a normal shift-invariant
conditional expectation onto $\mathcal{R}_{\varphi_\bz}^\tau$. The following computation shows that $F$  equals
$ E_\tau$ on all elements of the form $\pi_{\varphi_\bz}(a)$, with $a$ being a localized element of $\ast_\bn\ga$:
\begin{align*}
F(\pi_{\varphi_\bz}(a))&=F(\tau^n(\pi_{\varphi_\bz}(a)))=\lim_n F(\tau^n(\pi_{\varphi_\bz}(a)))=
F(\lim_n \tau^n(\pi_{\varphi_\bz}(a)))\\
&=F(E_\tau(\pi_{\varphi_\bz}(a)))=E_\tau(\pi_{\varphi_\bz}(a))\, .
\end{align*}
By normality this in turn  entails that $F=E_\tau$,  as the elements above span a weakly dense
$*$-subalgebra of $\mathcal{R}_{\varphi_\bz}$.\\
We  are left with the task of verifying that $ E_\tau$ is $\bj_\bn$ invariant.
Let $h$ be a fixed element of $\bj_\bn$.
By norm density it suffices to verify the desired equality on elements
of the form  $\pi_{\varphi_\bz}(a)$ with $a=i_{j_1}(a_1)\cdots i_{j_l}(a_l)$,
To this end, let $b, c$ localized elements of $\ast_\bz\ga$. We have
\begin{align*}
\langle E_\tau ((\pi_{\varphi_\bz}(\a_h( a)))\pi_{\varphi_\bz}(b)\xi_{\varphi_\bz}, \pi_{\varphi_\bz}(c)\xi_{\varphi_\bz}\rangle=\lim_n \varphi_\bz(c^*\tau^n(\a_h(a))b)\,.
\end{align*}
Now, as soon as $n$ is large enough, there exists a monotone map $g$ that is  the identity
on the supports of both $b$ and $c$ and sends  $j_k+n$  to $\a_h(j_k)+n$ for all $k=1, \ldots, l$.
By $\bj_\bz$-invariance of $\varphi_\bz$ we  eventually have
$\varphi_\bz(c^*\tau^n(\a_h(a))b)=\varphi_\bz(c^*\tau^n(a)b)$. But then:
\begin{align*}
&\langle E_\tau ((\pi_{\varphi_\bz}( \a_h(a)))\pi_{\varphi_\bz}(b)\xi_{\varphi_\bz}, \pi_{\varphi_\bz}(c)\xi_{\varphi_\bz}\rangle=\lim_n \varphi_\bz(c^*\tau^n(a)b)\\
=&\lim_n\langle (\pi_{\varphi_\bz}(\tau^n(a)))\pi_{\varphi_\bz}(b)\xi_{\varphi_\bz}, \pi_{\varphi_\bz}(c)\xi_{\varphi_\bz}\rangle=
\langle E_\tau (\pi_{\varphi_\bz}( a))\pi_{\varphi_\bz}(b)\xi_{\varphi_\bz}, \pi_{\varphi_\bz}(c)\xi_{\varphi_\bz}\rangle\,.
\end{align*}
hence $E_\tau ((\pi_{\varphi_\bz}(\a_h( a)))=E_\tau (\pi_{\varphi_\bz}( a))$ by cyclicity of the GNS vector.\\
Finally, verifying that $E_\tau$ is also permutation-invariant when $\varphi$ is exchangeable can be done in the same way as above.
\end{proof}
If $\mathcal{R}$ is a von Neumann algebra, we denote by $\cs_*(\mathcal{R})$ the convex set of all normal states of $\mathcal{R}$.
In the sequel, $\mathcal{T}_{\varphi_\bz}=\mathcal{R}_{\varphi_\bz}^\tau$ and $\mathcal{T}_\varphi$ are identified via the restriction isomorphism. We define $\cs_{\pi_{\varphi_\bz}}$ to be the convex set of all states of $\ast_\bn\ga$ which are normal in the representation $\pi_{\varphi_\bz}$.
As a consequence of the existence of the conditional expectation onto the tail algebra, we find that the tail algebra itself encodes  the face 
$\cs_{\pi_{\varphi_\bz}}^\tau:=\{\eta\in\cs_{\pi_{\varphi_\bz}}: \eta\circ\tau=\eta \,\,\}$ of distributions of all
 stationary processes normal in the GNS of the given extended process to $\bz$. 

\begin{cor}\label{normalinvspread}
Let $\varphi$ be a spreadable (exchangeable) state on $\ast_\bn\ga$. The map
$$\cs_*(\mathcal{T}_\varphi)\ni\om\mapsto \om\circ E_\tau\circ\pi_{\varphi_\bz}\in \cs_{\pi_{\varphi_\bz}}^\tau$$
is an affine homeomorphism. In particular,  the states in  $\cs^\tau_{\pi_{\varphi_\bz}}$ are automatically spreadable (exchangeable).
\end{cor}

For any state $\varphi$ on $\ast_\bn\ga$, we denote by $\widetilde{\varphi}$ the vector state (as a state on the whole $\cb(\ch_\varphi)$) associated with the GNS vector $\xi_\varphi$ of the GNS representation $(\ch_\varphi, \xi_\varphi, \pi_\varphi)$ of $\varphi$, that is $\widetilde{\varphi}(T):=\langle T\xi_\varphi ,\xi_\varphi \rangle$, $T$ in $\cb(\ch_\varphi)$.
The following results completes the partial characterization of extremality given in \cite[Proposition 5.4.3]{DKW} in terms of pureness of the restriction to the tail algebra.

\begin{cor}\label{purerestriction}
A spreadable (resp. exchangeable) state on $\ast_\bn\ga$ is extreme in $\cs^{\bj_\bn}(\ast_\bn\ga)$ (resp. in $\cs^{\bp_\bn}(\ast_\bn\ga)$), or equivalently in $\cs^{\tau}(\ast_\bn\ga)$, 
if and only if the restriction of $\widetilde{\varphi}$ to the tail algebra is extreme in $\cs_*(\mathcal{T}_\varphi)$ ({\it i.e.} a pure and normal state of $\mathcal{T}_\varphi$).
\end{cor}
\begin{proof}
Note that the fact that extremality in $\cs^{\bj_\bn}(\ast_\bn\ga)$  and in $\cs^{\tau}(\ast_\bn\ga)$ (resp. or in $\cs^{\bp_\bn}(\ast_\bn\ga)$) agree, is the content of Corollary \ref{faces}.

We will only deal with the exchangeable case, for the spreadable case can be handled in the same way.
As the tail algebra $\mathcal{T}_\varphi$ is identified with $\mathcal{R}^\tau_{\varphi_\bz}$, we may as well work in the GNS representation
of $\varphi_\bz$ if needs be.\\
 If $\widetilde{\varphi_\bz}$ fails to be extreme in $\cs_*(\mathcal{R}^\tau_{\varphi_\bz})$, then 
$\widetilde{\varphi_\bz}\upharpoonright_{\mathcal{R}^\tau_{\varphi_\bz}}=t\om_1+(1-t)\om_2$ for some $t\in (0,1)$ and
$\om_1\neq\om_2$ normal states on $\mathcal{R}^\tau_{\varphi_\bz}$.
The extensions $\widetilde{\om_i}=\om_i\circ E_\tau$, $i=1, 2$, to $\mathcal{R}_{\varphi_\bz}$
are exchangeable normal states of $\mathcal{R}_{\varphi_\bz}$
and $\widetilde{\varphi_\bz}=t\widetilde{\om_1}+(1-t)\widetilde{\om_2}$ (since $\widetilde{\varphi_\bz}$ is uniquely determined by its restriction to the tail algebra). This shows that $\varphi=t\varphi_1+(1-t)\varphi_2$, with
$\varphi_i(x)=\widetilde{\om_i}(\pi_{\varphi_\bz}(x))$, $x$ in $\ast_\bn\ga$. Now $\varphi_1\neq \varphi_2$, because
by normality $\widetilde{\om_1}$ and  $\widetilde{\om_2}$ are completely determined by their restriction to the weakly dense 
$^*$-subalgebra $\pi_{\varphi_\bz}(\ast_\bn\ga)$. Therefore, $\varphi$ is not extreme in $\cs^{\bp_\bn}(\ast_\bn\ga)$.\\
Conversely, if $\varphi$ is not extreme in $\cs^{\bp_\bn}(\ast_\bn\ga)$, then it dominates some exchangeable state $\om$ with
$\om\neq \varphi$, say $\om(x^*x)\leq M\varphi(x^*x)$, for all $x$ in $\ast_\bn\ga$, for some $M\geq 1$. By the non-commutative Radon-Nikodym theorem, $\om$ can be realized as a vector state, $\widetilde{\om}$, in the GNS representation
of $\varphi$, and $\widetilde{\om}\leq M\widetilde{\varphi}$ on the positive cone of $\mathcal{R}_\varphi$ as well. Since exchangeable states which are normal in the GNS representation of $\varphi$ are uniquely determined by their restriction to $\mathcal{T}_\varphi$,
the restriction of $\widetilde{\om}$ to the tail algebra is still different from $\widetilde{\varphi}\upharpoonright_{\mathcal{T}_\varphi}$, meaning that  the latter state cannot be extreme among all normal states of the tail algebra.
\end{proof}


\section {Non-commutative Hewitt-Savage $0-1$ law}
\label{01law}

The present section is devoted to characterizing when the tail algebra of a spreadable state is trivial. As explained
in the introduction, this is done in a twofold fashion in terms of a dynamical condition, attractiveness, and a factorization property of the state, forward-decoupling condition, which we next define. To do so, let us consider the convex set of all states which are normal in a given representation $\pi$ of the free product $\ast_\bn\ga$,
$$\cs_{\pi}:=\{\om\in\cs(\ast_\bn\ga):\ \om(\cdot)= {\rm tr}(\pi(\cdot)T)\, \text{with}\, T\in\cb(\ch_\pi)_+,\,{\rm tr}(T)=1\}\, .$$

\begin{defin} \label{equi}
A stationary state $\varphi$ on $\ast_\bn\ga$ is said to  be:
\begin{itemize}
\item[-]
$\pi$-attractive ($\pi$ being a representation of $\ast_\bn\ga$) if for every
$\om$ in $\cs_\pi$ one has
\begin{equation}\label{equilibrium}
\lim_{n\rightarrow\infty}    \om(\tau^n(x))=\varphi(x)\,\, \textrm{for all}\, x\in\ast_\bn\ga\,. 
\end{equation}
\item[-] forward-decoupling if for all localized $x, y, z$ in $\ast_\bn\ga$ one has
\begin{equation}\label{fdecoupling}
\varphi(xyz)=\varphi(zx)\varphi(y),
\end{equation}
provided that $x,y,z$ is have supports contained in finite subsets $F_x, F_y, F_z$ respectively, in the sense of Definition \ref{localized},  with $F_y > F_x \cup F_z$.
\item[-] block-singleton if  for all localized $x, y, z$ in $\ast_\bn\ga$ one has
\begin{equation}\label{blocksingleton}
\varphi(xyz)=\varphi(zx)\varphi(y)
\end{equation}
provided that $x,y,z$ is have supports contained in finite subsets $F_x, F_y, F_z$ respectively, in the sense of Definition \ref{localized},  with $F_y \cap (F_x \cup F_z)=\emptyset$.
\end{itemize}
\end{defin}
Note that for exchangeable states being forward-decoupling or block-singleton is the same property.\\

\begin{rem}
The block-singleton condition was introduced in \cite{ACGL}, which highlighted its relevance  as a form  of non-commutative   independence. However, several statements  in that paper do not hold in full generality. Since  
the block-singleton condition plays a relevant role in our results as well, we shall comment on these technical 
 issues in Remarks \ref{nocentralsupport} and \ref{blocknoteqprod} and provide explicit counterexamples.
\end{rem}

What follows is exactly the non-commutative Hewitt-Savage  0--1 law.
\begin{thm}\label{01lawspread}
Let $\varphi$ be a spreadable (resp. exchangeable) state  on $\ast_\bn\ga$. The following are equivalent:
\begin{enumerate}
\item[(i)] $\mathcal{T}_\varphi=\mathbb{C}$;
\item[(ii)] $\varphi$ is forward-decoupling (resp. is block-singleton);
\item[(iii)] $\varphi$ is $\pi_{\varphi_\bz}$-attractive.
 \end{enumerate}
\end{thm}

\begin{proof}
We will prove the implications $(ii)\Rightarrow (i)\Rightarrow (iii)\Rightarrow (ii)$.\\
We start by proving that any forward-decoupling state has  a trivial tail algebra.
We claim that for any such state $\varphi$ 
the GNS vector $\xi_{\varphi}$ is separating for $\mathcal{T}_{\varphi}$.\\
Indeed, suppose that
$T$ in $\mathcal{T}_{\varphi}$ vanishes on $\xi_{\varphi}$, 
For any localized element $x$ of
$\ast_\bn\ga$, say with support contained in $\{0, \ldots, l\}$, we have
\begin{align*}
\|T\pi_{\varphi}(x)\xi_{\varphi}\|^2&= \langle T\pi_{\varphi}(x)\xi_{\varphi},T\pi_{\varphi}(x)\xi_{\varphi} \rangle =
\langle \pi_{\varphi}(x^*)T^*T\pi_{\varphi}(x)\xi_{\varphi} ,\xi_{\varphi}  \rangle\\
&= \langle T^*T\xi_{\varphi} ,\xi_{\varphi}  \rangle \langle  \pi_{\varphi}(x^*x)\xi_{\varphi} ,\xi_{\varphi}  \rangle=0\,,
\end{align*}
since the support of $T$ is certainly to the right of $l$.\\
We next show that knowing that $\xi_\varphi$ is separating for $\mathcal{T}_\varphi$ allows us to
inject $\mathcal{T}_{\varphi}$ into 
$E_{\varphi}\mathcal{R}_{\varphi} E_{\varphi}$ via the linear map $\mathcal{T}_\varphi \ni T\mapsto E_\varphi T E_\varphi$.
To this aim, we start by showing that for all $A, B$ in $\mathcal{T}_{\varphi}$, we have 
\begin{equation}\label{projinbetween}
E_{\varphi} A E_{\varphi} B E_{\varphi}=E_{\varphi}AB E_{\varphi}\,.
\end{equation}
In order to prove this, we need to take advantage of the representation of $E_{\varphi}$ as the strong limit $E_{\varphi} =\lim_n \frac{1}{n}\sum_{k=0}^{n-1}V_\tau^k$. We then have:
\begin{align*}
 &E_{\varphi} A E_{\varphi}  B  E_{\varphi} =\lim_n E_{\varphi} A \frac{1}{n}\sum_{k=0}^{n-1}V_\tau^kBE_{\varphi}=\lim_n E_\varphi A\sum_{k=0}^{n-1}  \frac{1}{n}B E_{\varphi}\\
&=E_{\varphi}ABE_{\varphi}\,,
\end{align*}
where we have exploited the equality $V_\tau^kBE_{\varphi}=BE_{\varphi}$ for all natural $k$.\\
This in turn can be seen in the following way. By Theorem \ref{HSonZ}, there exists
a unique $\widetilde{B}$ in $\mathcal{R}_{\varphi_{\bz }}^{\tau}$  whose restriction to $\ch_\bn\cong\ch_\varphi$ 
coincides with $B$. By restricting the operator equality $U_\tau^k\widetilde{B} (U_\tau^k)^*=\widetilde{B}$
to vectors of the range of $E_{\varphi}$, one gets the very equality $V_\tau^kBE_{\varphi}=BE_{\varphi}$ we needed to verify.\\
We are now ready to get to the conclusion.
If $T$ in $\mathcal{T}_{\varphi}$  is such that   $E_{\varphi}T E_{\varphi}=0$, then 
 $E_{\varphi}T^*T E_{\varphi}=0$ thanks to Equation \eqref{projinbetween}. Therefore, we have
$TE_{\varphi}=0$, hence $T\xi_{\varphi}=0$, which gives $T=0$ because $\xi_{\varphi}$ is separating  for 
$\mathcal{T}_\varphi$.\\
Now  $ E_{\varphi}$ is the projection onto $\bc\xi_{\varphi}$ by Proposition \ref{ergospread} because the forward-decoupling condition  implies $(iii)$ of the mentioned proposition.
As a consequence of $E_\varphi$ being one-dimensional, we finally see that
$\mathcal{T}_\varphi$ is trivial.\\
As for the implication $(i)\Rightarrow (iii)$, if $\mathcal{T}_\varphi$ is trivial, the conditional expectation $E_\tau$ constructed in Theorem \ref{condexpgen} projects onto $\bc$ and thus coincides with the normal extension of $\varphi_{\bz }$ on $\mathcal{R}_{\varphi_\bz}$.
Let $\om$ be a state in the folium $\cs_{\pi_{\varphi_\bz}}$. For any localized $x$ in
$\ast_\bn\ga$, by Theorem \ref{condexpgen}, we have $\lim_n\om(\tau^n(x))=\tilde{\om} (E_\tau(\pi_{\varphi_\bz}(x)))=\varphi(x).\, $\\
Finally, we prove that $(iii)\Rightarrow (ii)$.\\
Let $x$, $y$, $z$ be localized in $\ast_{\bz }\ga$, with the support of $y$ to the right of the supports of both $x$ and $z$.
We start by observing that a simple polarization argument shows that it suffices to verify the forward-decoupling condition when $z= x^*$, for the general case of $z$ not necessarily equal to $x^*$ is got to by considering $\varphi((x+z)y(x^*+z^*))$.\\
Now if $\varphi(xx^*)=0$, then  by the Cauchy-Schwartz inequality we have $|\varphi(xyx^*)|\leq \varphi(xx^*)^{\frac{1}{2}} \varphi(xy^*yx^*)^\frac{1}{2}=0$, hence $\varphi(xyx^*)=0=\varphi(xx^*)\varphi(y)$ and there is nothing left to prove. Therefore, we can safely suppose
that $\varphi(xx^*)$ does not vanish.
If $(iii)$ holds,  we can apply the attractivity condition to the state
$\omega(y):=\frac{1}{\varphi_{\bz }(xx^*)}\varphi_{\bz }(xyx^*)$, which is of course normal in the GNS representation of
$\varphi_\bz$, obtaining
 \begin{align*} 
&\lim_{n\rightarrow\infty} \frac{1}{\varphi_{\bz }(xx^*)} \varphi_{\bz }(x\tau^n(y)x^*)=
\varphi_{\bz }(y)\,,
\end{align*}
that is
  \begin{align*} 
& \lim_{n\rightarrow\infty}\varphi_{\bz }(x\tau^n(y)x^*)=
\varphi_{\bz }(y) \varphi_{\bz }(xx^*)\,.
\end{align*}
On the other hand, for $n$ big enough
there exists a monotone map in $\mathbb{J}_\bz$ that is
the identity on the supports of $x$  and 
acts as
the shift by $n$
on the support of $y$. Therefore, by $\mathbb{J}_\bz$-invariance of 
$\varphi_{\bz }$, we have
$\ \varphi_{\bz }(xyx^*)=   \varphi_{\bz }(x\tau^n(y)x^*)$, hence
$\ \varphi_{\bz }(xyx^*)= \varphi_{\bz }(xx^*)  \varphi_{\bz }(y)$, which is exactly what we meant to prove.
\end{proof}

We now present some notable examples to illustrate the above result where specific non-commutative independeces are
introduced.

\begin{defin}\label{freetens}
Following \cite{HO}, a quantum stochastic process with distribution  $\varphi$   is said to be:
\begin{itemize} 
\item[-]   freely independent if

$$
\varphi(\iota_{i_1}(a_1) \cdots \iota_{i_m}(a_m))=0
$$
when \(i_1 \neq i_2 \neq \cdots \neq i_m\) and $\varphi(\iota_{i_j}(a_j)) = 0$ for all $j=1, 2, \ldots m$;
\item[-] tensor independent if
$$\varphi(\iota_{i_1}(a_1) \cdots \iota_{i_m}(a_m))=\varphi(\iota_{i_1}(a_1))\varphi(\iota_{i_2}(a_2) \cdots \iota_{i_m}(a_m))$$ when $i_1$ 
is not in  $\{i_k\}_{k=2}^m$, or
$$\varphi(\iota_{i_1}(a_1) \cdots \iota_{i_m}(a_m))=\varphi(\iota_{i_2}(a_2) \cdots \iota_{i_1}(a_{r-1}) (\iota_{i_1}(a_1a_r)) \iota_{i_m}(a_m))$$ 
when $i_1=i_r$.
\end{itemize}
\end{defin}

\begin{cor}\label{freeprod}
The distribution $\varphi$ of a quantum stochastic process which is identically distributed and either freely or tensor independent
satisfies the three equivalent conditions in Theorem \ref{01lawspread}. 
\end{cor}

\begin{proof}
For the free case, note that in order for $\varphi$ to be identically distributed and freely independent, it is necessary and sufficient that
$\varphi=\ast\om$, the infinite free product of some state $\om$ on the sample algebra $\ga$, see \cite{A}. Now, the conclusion follows from the fact that free product states are block-singleton, {\it cf.} \cite[Proposition 2.2]{ACGL}.\\
As for the tensor indepedence case, it is a trivial computation that a tensor independent state is also block-singleton.
\end{proof}

\begin{rem}
\begin{itemize}

\medskip
\item[-] Boolean, monotone, anti-monotone independence will be addressed in Section 7 as these forms of independece need to be set in the framework of non-unital processes.
\item[-] Further classes of processes which satisfy the thesis of the Hewitt-Savage 0--1 law but do not satisfy any of
the five Muraki's natural forms of non-commutative independence are supplied by twisted tensor products, see Remark 8.2.
\item[-] An example of a state that is forward-decoupling while not being block-singleton is given in  Example \ref{monotonetail}.
\end{itemize}
\end{rem}

Along the way, we also provide an application of the dynamical condition in Theorem \ref{01lawspread} to general representation theory 
of infinite free product $C^*$-algebras, which represents a free analogue of \cite[Theorem 6.1]{Sto}.

\begin{cor}
Two free product states  $\ast_\bn\om$ and $\ast_\bn\varphi$ on $\ast_\bn\ga$ are quasi-equivalent if and only if $\om=\varphi$.

\end{cor}
\begin{proof}
We first observe that one may as well work with the infinite free product indexed by $\bz$.
 Indeed, by universality of free products any bijection
$f:\bn\rightarrow \bz$ can be lifted to a $*$-isomorphism $\Psi_f:\ast_\bn\ga\rightarrow\ast_\bz\ga$. Since 
$\Psi_f\circ\iota_n=\iota_{f(n)}$ for all $n$ in $\bn$, we have $\ast_\bz\om\circ\Psi_f=  \ast_\bn\om$, from which it follows that quasi-equivalence of $\ast_\bn\om$ and $\ast_\bn\varphi$ is the same  as  quasi-equivalence of $\ast_\bz\om$ and $\ast_\bz\varphi$.\\
That said, let $\om, \varphi$ be states on $\ga$ such that the GNS representations of 
$\ast_\bz\om$ and $\ast_\bz\varphi$ are quasi-equivalent.
In particular, $\ast_\bz\om$ is  normal in the GNS representation of $\ast_\bz\varphi$, but then  $(\ast_\bz\om)\,\circ\tau^n=\ast_\bz\om$ must converge to
$\ast_\bz\varphi$, that is $\ast_\bz\om=\ast_\bz\varphi$, that is $\om=\varphi$.
\end{proof}

\subsection{Hierarchy of ergodicities}

In this section what we called a hierarchy of ergodicities is established. What we do is exhibit four increasingly stronger levels of ergodicity for states on the infinite free product, ranging from mere extremality to extremality with central support, and, importantly, including the Hewitt-Savage 0--1 law.\\ 
In the  setting of Classical Probability, all these levels collapse to the same notion since states on commutative
$C^*$-algebras always have central support. Interestingly, the hypothesis of central support (or rather of having a von Neumann algebra with a normal faithful state) is the basis of much literature in the field, see {\it e.g.} \cite{BCS, GK08, K, KS}.\\
Furthermore, each level of ergodicity is characterized in a number of equivalent ways. 
Finally, in the exchangeable case the fourth and strongest level corresponds to having a trivial exchangeable algebra in the GNS representation of the state.\\

First of all we recall the following.
\begin{defin}
Let $\varphi$ be a state on $\ast_\bn\ga$. $\varphi$ is said to 
\begin{itemize}
\item[-]  be \emph{strongly clustering} if 
$$\lim_{n\rightarrow\infty} \varphi(x\tau^n(y))=\varphi(x)\varphi(y),$$  for  all $x, y$ in $\ast_\bn\ga$;
\item[-] satisfy the \emph{full product condition} if 
$$\varphi(xy)=\varphi(x)\varphi(y),$$ for  all localized $x, y$ in $\ast_\bn\ga$ supported in subsets $F_x, F_y$ with $F_x \cap F_y =\emptyset$;
\item[-] satisfy the \emph{ordered product condition} if 
$$\varphi(xy)=\varphi(x)\varphi(y),$$ for  all localized $x, y$ in $\ast_\bn\ga$ supported in subsets $F_x, F_y$ with either $F_x < F_y$ or $F_x > F_y$.
\end{itemize}
\end{defin}
\begin{rem}

For an exchangeable state the two notions of being an ordered product state or a full product state trivially agree.\\
We would like to point out that $\pi_{\varphi_\bz}$-attractivity is a stronger form of strongly clustering, for it is equivalent to
 $
 \lim_{n\rightarrow\infty}\varphi(x^*\tau^n(y)z)=
\varphi(y) \varphi(x^*z)\, ,
$
for all $x, y, z$ in $\ast_\bn\ga$, {\it cf.} the proof of Theorem \ref{01lawspread}.\\
We end this remark by noting  that the full/ordered product conditions are natural extensions of the factorization of momenta in the classical setting. It is well-known that all forms of non-commutativity satisfy these two properties.
\end{rem}
We start with the hierarchy for the exchangeable case. The spreadable one is provided in Theorem 	\ref{ergohierspread}.

\begin{thm}\label{ergohierexchange}
Let $\varphi$ be an exchangeable state on $\ast_\bn\ga$. Then:

\begin{itemize}
\item [(i)] $\varphi$ is extreme in $\cs^{\bp_\bn}(\ast_\bn\ga)$  $\Leftrightarrow$ $\{\pi_\varphi (\ast_\bn\ga), U_\s^\varphi: \s\in\bp_\bn \}'=\mathbb{C}$;
\item [(ii)] ${\rm dim}\,E_\varphi=1$ $\Leftrightarrow$ $\varphi$ is strongly clustering $\Leftrightarrow$ $\varphi$ is a full product state
\item [(iii)] $\mathcal{T}_\varphi=\bc$ $\Leftrightarrow$ $\varphi$ is block-singleton $\Leftrightarrow$ $\varphi$ is $\pi_{\varphi_{\bz}}$-attractive;
\item [(iv)] $\mathcal{R}_\varphi^{\bp_\bn}=\bc$  $\Leftrightarrow$ $\varphi$ is extreme in $\cs^{\bp_\bn}(\ast_\bn\ga)$ and has central support.
\end{itemize}
Furthermore, (iv) $\Rightarrow$ (iii) $\Rightarrow$ (ii) $\Rightarrow$ (i).
\end{thm}

\begin{proof}
The equivalence in $(i)$ is a well-known general fact, see \emph{e.g.} \cite[Theorem 4.3.20]{BR1}.
The equivalences in $(ii)$ are the content of Proposition \ref{ergospread}. The equivalences in $(iii)$ are Theorem \ref{01lawspread}.\\
The equivalence in $(iv)$ has to be proved. In one direction,  if $\mathcal{R}_\varphi^{\bp_\bn}=\bc$, then $\xi_\varphi$ is a separating vector
for $\mathcal{R}_\varphi$ by \cite[Theorem 4.3.20]{BR1} and $\ch_\varphi^{\bp_\bn}=\overline{ \mathcal{T}_\varphi \xi_\varphi}\subseteq \overline{\mathcal{R}_\varphi^{\bp_\bn}\xi_\varphi}=\bc\xi_\varphi$ by Theorem \ref{HSspreadonL2}, so $E_\varphi$ is a one-dimensional projection and {\it a fortiori} $\varphi$ is  extreme by {\it e.g.} \cite[Theorem 4.3.20]{BR1}.
The other implication follows again by \cite[Theorem 4.3.20]{BR1}.\\
The implication $(iv)\Rightarrow (iii)$ can be proved in the following way. The restriction of $\widetilde{\varphi}$, the vector state on $\cb(\ch_\varphi)$ corresponding to $\xi_\varphi$, to $\mathcal{T}_\varphi$ is a faithful normal state, which is also pure by
Corollary \ref{purerestriction}. Let $\ch_{\widetilde{\varphi}}$ denote the GNS Hilbert space of $(\mathcal{T}_\varphi,\widetilde{\varphi})$. Then $\mathcal{T}_\varphi$ is $*$-isomorphic with $\cb(\ch_{\widetilde{\varphi}})$, and
$\ch_{\widetilde{\varphi}}$ must be a one-dimensional Hilbert space since $\cb(\ch_{\widetilde{\varphi}})$ has a separating vector, the GNS vector of $\widetilde{\varphi}$.\\
The implication $(iii)\Rightarrow (ii)$ is trivial as the block singleton property clearly implies the full product property.\\
Finally, the implication $(ii)\Rightarrow (i)$ is a well-known general fact, \cite[Theorem 4.3.20]{BR1}.
\end{proof}

\begin{rem}\label{nocentralsupport}
In the light of the previous result, the validity of a strong form of the Hewitt-Savage law that in the GNS representation $\pi_\varphi$ of an exchangeable state $\varphi$ on $\ast_\bn \ga$ the exchangeable algebra $\mathcal{R}_{\varphi}^{\bp_\bn}$ and the tail algebra
are both trivial only holds for states with central support. \\
The notion of central support also clarifies the scope of the following result from the literature:
Theorem 3.2 in \cite{ACGL} claims the equivalence for an exchangeable state on $\ast_\bz \ga$ between having trivial stationary algebra (precisely $\{T\in\pi_{\varphi_\bz}(\ast_\bz\ga)'': U_\tau T U_\tau^*=T\}$ is trivial)
and satisfying the block-singleton condition. However, the first condition implies that the state has central support by \cite[Theorem 4.3.20]{BR1}, whereas no such consequence on the support can be drawn from the block-singleton condition, see e.g. Example \ref{tensprodex}.
\end{rem}

The following is the full hierarchy for spreadable distributions.

\begin{thm}\label{ergohierspread}
Let $\varphi$ be a spreadable state on $\ast_\bn\ga$. Then:

\begin{itemize}
\item [(i)] $\varphi$ is extreme in $\cs^{\bj_\bn}(\ast_\bn\ga)$ $\Leftrightarrow$  $\varphi$ is extreme in $\cs^{\tau}(\ast_\bn\ga)$ $\Leftrightarrow$ $\{\pi_\varphi (\ast_\bn\ga), V_\tau  \}'=\mathbb{C}$;
\item [(ii)] ${\rm dim}\,E_\varphi=1$ $\Leftrightarrow$ $\varphi$ is strongly clustering $\Leftrightarrow$ $\varphi$ is an ordered product state
$\Leftrightarrow$ $\varphi$ is a full product state;
\item [(iii)] $\mathcal{T}_\varphi=\bc$ $\Leftrightarrow$ $\varphi$ is forward-decoupling $\Leftrightarrow$ $\varphi$ is $\pi_{\varphi_\bz}$-attractive;
\item [(iv)]   $\varphi$ is extreme in $\cs^{\bj_\bn}(\ast_\bn\ga)$ and has central support. 
\end{itemize}
Furthermore, (iv) $\Rightarrow$ (iii) $\Rightarrow$ (ii) $\Rightarrow$ (i).
\end{thm}

\begin{proof}
For the equivalences in $(i)$ note that by Corollary \ref{faces} extremality in  $\cs^{\bj_\bn}(\ast_\bn\ga)$ is the same as extremality
in $\cs^{\tau}(\ast_\bn\ga)$.\\
We can then move on to show that $\varphi$ is extreme in $\cs^{\tau}(\ast_\bn\ga)$ $\Leftrightarrow$ $\{\pi_\varphi (\ast_\bn\ga), V_\tau  \}'=\mathbb{C}$. The implication $\varphi$ is extreme in $\cs^{\tau}(\ast_\bn\ga)$ $\Rightarrow$ $\{\pi_\varphi (\ast_\bn\ga), V_\tau  \}'=\mathbb{C}$ is straightforward (for any non-scalar positive $T$ in the commutant $\{\pi_\varphi (\ast_\bn\ga), V_\tau  \}'$ gives rise to a state properly dominated by $\varphi$).\\
For the reverse implication, if $\om$ is dominated by $\varphi$, then
$\om(x)=\langle \pi_\varphi(x)T\xi_\varphi, \xi_\varphi \rangle$, $x$ in $\ast_\bn\ga$, for a suitable
positive $T$ in $\pi_\varphi(\ast_\bn\ga)'$ by the non-commutative Radon-Nikodym theorem. Therefore, the conclusion will be reached as soon as we showh that $T$ commutes with $V_\tau$. By the same argument as in the proof
of Corollary \ref{faces} one sees that $V_\tau T\xi_\varphi= T\xi_\varphi$.
Now the equality $V_\tau\pi_\varphi(x)\eta= \pi_\varphi(\tau(x))\eta$ holds for all $x$ in $\ast_\bn\ga$, if $\eta$ is an invariant vector (that is if $V_\tau\eta=\eta$ ). Indeed, let $\{x_n:n\in\bn\}$ be a sequence in $\ast_\bn\ga$ such that $\eta=\lim_n \pi_\varphi(x_n)\xi_\varphi$.
Because $\eta$ is invariant, we also have $\eta=\lim_n \pi_\varphi(\tau(x_n))\xi_\varphi$. But then
\begin{align*}
V_\tau\pi_\varphi(x)\eta&=\lim_n V_\tau \pi_\varphi(xx_n)\xi_\varphi=\lim_n \pi_\varphi(\tau(x))\pi_\varphi(\tau(x_n))\xi_\varphi\\
&=\pi_\varphi(\tau(x))\eta\, .
\end{align*}
The sought commutation between $T$ and the isometrie $V_\tau$ can now be derived easily in the following way. For all
$x$ in $\ast_\bn\ga$, we have:
\begin{align*}
V_\tau T\pi_\varphi(x)\xi_\varphi&= V_\tau\pi_\varphi(x) T\xi_\varphi=\pi_\varphi(\tau(x))T\xi_\varphi=
T\pi_\varphi(\tau(x))\xi_\varphi\\
&=TV_\tau \pi_\varphi(x)\xi_\varphi\, ,
\end{align*}
which shows $V_\tau T=T V_\tau$ as $\pi_\varphi(\ast_\bn\ga)\xi_\varphi$ is a dense subspace of $\ch_\varphi$.\\
The equivalence of the first three conditions in (ii) is the content of Proposition \ref{ergospread}. The full product state condition obviously implies the ordered product state condition. The converse implication is non-trivial and requires the discussion in the appendix culminating in Lemma \ref{insertproj}. Using Lemma \ref{insertproj}, the full product state condition follows immediately from the condition 
${\rm dim}\, E_\varphi=1$.\\
The equivalences in $(iii)$ have been proved in Theorem \ref{01lawspread}.\\
The chain of implications (iv) $\Rightarrow$ (iii) $\Rightarrow$ (ii)  can be proved in the exact 
same way as we did Theorem  \ref{ergohierexchange}. The implication (ii) $\Rightarrow$ (i) requires further discussion because the one in Theorem \ref{ergohierexchange} relied on the fact that  $\bp_\bn$ is a group acting by automorphisms.
By Corollary \ref{purerestriction}, showing that $\varphi$ is extreme in $\cs^{\bj_\bn}(\ast_\bn\ga)$ amounts to show that $\widetilde{\varphi}$  is a pure state on $\mathcal{T}_\varphi$,
where $\widetilde{\varphi}(T):=\langle T\xi_\varphi ,\xi_\varphi \rangle$ for  $T\in \mathcal{T}_\varphi$. Theorem \ref{HSspreadonL2} implies that $E_\varphi \ch_\varphi$ is an invariant space for $\mathcal{T}_\varphi$ and that the GNS representation of $(\mathcal{T}_\varphi, \widetilde{\varphi} )$ is given by the action of $\mathcal{T}_\varphi$ on $E_\varphi \ch_\varphi$. Consequently, by the hypothesis in (ii), $E_\varphi \ch_\varphi$ is one-dimensional, meaning that $\widetilde{\varphi}$ is a pure state.
\end{proof}

\begin{rem}
None of the implications (iv) $\Rightarrow$ (iii) $\Rightarrow$ (ii) $\Rightarrow$ (i) is an equivalence.\\
The states associated with constant processes that we discuss below in Example \ref{constantprocess} are extreme
but ${\rm dim} E_\varphi$ can be as large as wished as long as a suitable sample algebra is chosen, meaning (i) does not imply (ii).
As we shall see in Section \ref{models}, product states on twisted tensor products  which are not gauge-invariant
feature a non-trivial tail algebra, see Remark \ref{tailGinv} and Theorem \ref{isotailtorus}, meaning (ii) does not imply (iii).\\
Finally, Example \ref{tensprodex} 
showcases that (iii) does not imply (iv).
\end{rem}

\begin{example}\label{constantprocess}
Let $\ga$ be any non-commutative unital $C^*$-algebra and let $\om$ be a pure state of $\ga$ whose (irreducible) GNS representation
acts on a Hilbert space $\ch_\om$ with ${\rm dim} \ch_\om\geq 2$. Consider the corresponding state $\varphi$ on  the free product $\ast_\bn\ga$  determined by
$$\varphi(i_{i_1}(a_1)\cdots i_{i_k}(a_k))=\langle\pi_\om(a_1)\cdots \pi_\om(a_k) \xi_\om, \xi_\om\rangle\, $$
for all $k\in\bn, i_1, \ldots, i_k\in\bn, a_1, \ldots, a_k\in\ga$.\\
The state $\varphi$ is obviously exchangeable. Moreover, it is extreme since it is even pure, in that its GNS representation is irreducible.
Note that $E_\varphi= I$ in $\cb(\ch_\om)$ and $\mathcal{T}_\varphi=\cb(\ch_\om)$ because, for each $n$ in $\bn$, we clearly have  $(\bigcap_{k\geq n} \iota_k(\ga))''=\pi_\om(\ga)''=\cb(\ch_\om)$.\\
Also note that  any such $\varphi$  as above is possibly the simplest  conceivable example of an extreme exchangeable state whose projection $E_\varphi$ is not one-dimensional.  
\end{example}
 In the light of the results we have gleaned so far it is suggestive to think of a spreadable process
satisfying any of the three equivalent conditions in our  Hewitt-Savage 0--1 law as an identically distributed, independent (i.i.d. for short) process.\\
This way of construing the i.i.d. condition becomes particularly sharp when applied to the five non-commutative natural forms of independence, which, as we will show over the paper, can be obtained by combining this i.i.d. condition with the appropriate specific distributional symmetry (w.r.t. the action/coaction of suitable compact groups/quantum groups and semigroups) or, for tensor independence, with the request that the distribution factorizes through a suitable quotient.
As we will see, these characterizations are got to by carefully relying on the literaure of the field.\\
In the remaining part of this section, tensor independence and free independence are dealt with.
We recall that the notion of a distribution factorizing through a quotient was given in Definition \ref{distri}.

\begin{prop}\label{tensor}
For a quantum stochastic process with sample algebra $\ga$  and with distribution $\varphi$, the following are equivalent:
\begin{itemize}
\item [(i)] the process is identically distributed and tensor independent;
\item [(ii)] $\varphi$ is spreadable, satisfies any of the three conditions of Theorem \ref{01lawspread}, and factorizes through the infinite maximal tensor product $\otimes_\bn\ga$.
\end{itemize} 
\end{prop}

\begin{proof}
We start by proving that (i) implies (ii). If (i) holds, then $\varphi$ is spreadable and by Corollary \ref{freeprod}  it satisfies 
the three conditions of the non-commutative Hewitt-Savage 0--1 law. All is left to do is make sure that $\varphi$ factorizes 
through the infinite tensor product, that is $\varphi(I)=0$ if $I\subset \ast_\bn\ga$ is the two-sided ideal such that
$\ast_\bn\ga/ I \cong\otimes_\bn\ga$, and a moment's reflection shows that this amounts to showing that
$\varphi ((xyz-yxz)^*(xyz-yxz))=0$ for all localized $x, y, z$ such that $x$ and $y$ have disjoint supports.
This equality can be derived by applying the definition of tensor independence, Definition \ref{freetens}, and making some  easy yet tedious calculations.\\
As for the converse implication,  a state $\varphi$ satisfying (ii) is of the form $\varphi=\varphi'\circ p_I$, where $\varphi'=\otimes\om$ is a  tensor product state of
$\otimes_\bn\ga$, for some state $\om$ on the sample algebra $\ga$, and tensor product states of the infinite tensor product are seen at once to be tensor independent.
\end{proof}

\begin{rem}\label{free}
As is known, the infinite free product is naturally coacted upon by  the compact quantum groups of  so-called quantum permutations.  
The corresponding invariant states are known as the quantum exchangeable states. These are the main focus of the influential paper
\cite{KS} insofar as the symmetry they enjoy enables one to characterize (conditional) free independence.
The scope of that paper was later widened in \cite{DKW} by allowing for not necessarily faithful states and for  sample algebras not necessarily singly generated.
For definitions and precise statements of the results we refer the reader to \cite{DKW}.\\
What we aim to stress here is that a state is identically distributed, free independent if and only if it is quantum symmetric and satisfies any of the three conditions in the non-commutative Hewitt-Savage 0--1 law, Theorem \ref{01lawspread}.
Indeed, one implication is the content of
Corollary \ref{freeprod} and the well-known fact that free product states $\ast_{\bn}\om$ are quantum symmetric.
For the other,  if the state is quantum symmetric, then by Theorem 6.1 in \cite{DKW} it is conditionally
free independent w.r.t. the tail algebra, which is trivial under our assumptions. Clearly, conditional free independence w.r.t. a trivial tail algebra is just free indepenence.
\end{rem}

\begin{rem}\label{pasttail}
We would like to point out that in the GNS representation of $(\ast_\bz \ga,\varphi_\bz)$ a past tail algebra
$\mathcal{T}_{\varphi_\bz}^-$ can also be considered by defining it as 
$$\mathcal{T}_{\varphi_\bz}^-:=\bigcap_{n\in\bn}\{\pi_{\varphi_\bz}(i_k(\ga)): k\leq -n\}''\, .$$
The analysis we carried out for the tail algebra can be conducted for the past tail algebra as well provided that
one works entirely in the GNS representation $\pi_{\varphi_\bz}$ (here thought of as a representation of $\ast_\bz \ga$). In particular, one can prove the existence of a unique normal conditional expectation
from  $\pi_{\varphi_\bz}(\ast_{k=n}^{ -\infty} \ga)''$ (for any fixed $n$)
onto the past tail algebra satisying the same factorization and  invariance properties as the conditional expectation  onto the tail algebra, see
Theorem \ref{deFinetti}.\\
Furthermore,  an analogue of the Hewitt-Savage 0--1 law can be formulated for the past tail algebra as well. Precisely, for a spreadable state
$\varphi$ on $\ast_\bn\ga$ the following are equivalent:
\begin{itemize}
\item[-] $\mathcal{T}_{\varphi_\bz}^{-}=\bc$;
\item[-] $\varphi$ is backwards-decoupling, namely  $\varphi(xyz)=\varphi(xz)\varphi(y)$ for all localized
$x, y, z$ in $\ast_\bn\ga$ provided that $x,y,z$ is have supports contained in finite subsets $F_x, F_y, F_z$ respectively, in the sense of Definition \ref{localized},  with $F_y < F_x \cup F_z$.;
\item[-] $\varphi$ is $\tau^{-1}$-$\pi_{\varphi_\bz}$-attractive, namely $\lim_{n\rightarrow+\infty}\om(\tau^{-n}(x))=\varphi_\bz(x)$, for all $x$ sitting in $\ast_\bn\ga\subset\ast_\bz\ga$ and for all
$\om$ in $\cs_{\pi_{\varphi_\bz}}$ (w.r.t. the "negative" $C^*$-algebra $\ast_{k=0}^{-\infty}\ga$).\\
\end{itemize}
These conditions make up yet another layer in the hierarchy of ergodicities, say [(iii)'], which 
is the same level as [(iii)], in the sense that it is implied by [(iv)] and implies [(ii)], while not
 being  equivalent to [(iii)] in Theorem \ref{ergohierspread}, see Example \ref{monotonetail}.\\
However, if $\varphi$ is an exchangeable state, it is easy to see that $\mathcal{T}_{\varphi_\bz}^{-}\cong
\mathcal{T}_\varphi$, with  [(iii)] and [(iii)'] becoming equivalent.\\
As well as the tail algebra, the past tail algebra can be used  to develop a fully-flegded  specular de Finetti theory, including
the corresponding version of the de Finetti theorem given in Theorem \ref{deFinetticonv}.
\end{rem}

\section{Non-commutative de Finetti theorem}\label{definettisec}

In this section we finally state what we deem is the most general version of the non-commutative de Finetti theorem. Theorem \ref{deFinetti}
provides an extension of K\"{o}stler's de Finetti theorem in \cite{K} for spreadable process by dropping
any assumption of faithfulness on  the distribution of the processes.
In  the same theorem we show that the conditional expectation onto the tail algebra of a spreadable process enjoys the analogous factorizing properties corresponding to (iii) in the hierarchy
of ergodicities, namely it is conditionally forward-decoupling.\\
Importantly, if we weaken the assumption of spreadability by requesting weak speadability, as defined in Definition \ref{weakspread}, while maintaining the forward-decoupling property for the conditional expectation,
we find in Theorem \ref{deFinetticonv} an accomplished characterization of weak spreadability, which ought to conceived of as a full non-commutative de Finetti theorem.
Finally, Corollary \ref{deFinettiquot} covers the case of processes whose distribution factorizes through quotients such that weak spreadability and spreadability actually coincide. In particular,
it returns de Finetti's classical theorem when applied to the infinite tensor product of a commutative sample algebra.

\begin{defin}
The conditional expectation $E_\tau:\mathcal{R}_{\varphi_\bz}\rightarrow \mathcal{T}_{\varphi_\bz}$ onto the tail algebra of a spreadable state $\varphi$ is said to satisfy:
\begin{itemize}
\item[-] the ordered product condition if 
$$E_\tau(XY)= E_\tau(X)E_\tau(Y)$$
for all localized $X, Y$ in $\mathcal{R}_{\varphi_\bz}$ such that $X,Y$ have supports contained in finite subsets $F_X, F_Y$ respectively, in the sense of Definition \ref{localized},  with $F_X > F_Y$ or $F_X<F_Y$;
\item[-]  the full product condition if 
$$E_\tau(XY)= E_\tau(X)E_\tau(Y)$$
for all localized $X, Y$ in $\mathcal{R}_{\varphi_\bz}$ such that $X,Y$ have supports contained in finite subsets $F_X, F_Y$ respectively with $F_X\cap F_Y=\emptyset$;
\item[-] the forward-decoupling condition if
$$E_\tau(XYZ)= E_\tau(XE_\tau(Y)Z)$$
for all localized $X, Y, Z$ in $\mathcal{R}_{\varphi_\bz}$ such that $X,Y,Z$ have supports contained in finite subsets $F_X, F_Y,F_Z$ respectively with $F_Y>F_X\cup F_Z$;
\item[-] the block-singleton condition if
$$E_\tau(XYZ)= E_\tau(XE_\tau(Y)Z)$$
for all localized $X, Y, Z$ in $\mathcal{R}_{\varphi_\bz}$ such that $X,Y,Z$ have supports contained in finite subsets $F_X, F_Y,F_Z$ respectively with $F_Y \cap (F_X\cup F_Z)=\emptyset$.
\end{itemize}

\end{defin}

%

\begin{rem}
It is easy to see that the forward-decoupling condition implies the ordered product condition. Furthermore,
the block-singleton condition implies both the full product condition and the forward-decoupling condition.

At this point a few words of comment are in order. What we called ordered product condition and full product condition usually  goes in the literature under the heading of 
ordered conditional independence and full conditional independence, respectively, see \cite{K}. However, following our guiding principle that stems from the hierarchy or ergodicities, Theorem \ref{ergohierspread}, 
it is the forward-decoupling condition  instead that should be regarded as a non-commutative form of conditional independence, thus we have changed the nomenclature accordingly.
Also note that the factorization properties of the conditional expectation represented by ordered and full product conditions, correspond to level (ii) in the hierarchy of ergodicities, which  for states is no guarantee
of triviality of the tail algebra.
Finally, we would like to point out that the Classical Probability setting  does not allow to differentiate between the ordered/full product properties and the forward-decoupling condition.
\end{rem}

%
%

\begin{thm}\label{deFinetti}
The conditional expectation $E_\tau$ onto the tail algebra of a spreadable state on $\ast_\bn\ga$
 enjoys the forward-decoupling  and full product conditions. Moreover, if the state is also exchangeable, then $E_\tau$ satisfies
the block-singleton condition.
\end{thm}
\begin{proof}
By normality of $E_\tau$  all the stated properties need only be verified  on the weakly dense
$*$-subalgebra $\pi_{\varphi_\bz}(\ast_\bn\ga)$. Let $x, y, z$ in $\ast_\bn\ga$ localized elements with the support of $y$ to the right of that of $x$ and $z$.
For every $n$ in $\bn$, there exists a map $h_n$ in $\bj_\bn$ which acts as the identity on the supports of $x$ and $z$ but shifts by $n$ that of $y$. By $\bj_\bn$-invariance of $E_\tau$ we then have:
\begin{align*}
&E_\tau(\pi_{\varphi_\bz}(xyz))=E_\tau((\pi_{\varphi_\bz}(\a_{h_n}(x))\pi_{\varphi_\bz}(\a_{h_n}(y))\pi_{\varphi_\bz}(\a_{h_n}(z))))\\
&=E_\tau(\pi_{\varphi_\bz}(x)\pi_{\varphi_\bz}(\tau^n(y))\pi_{\varphi_\bz}(z))
=\lim_n E_\tau(\pi_{\varphi_\bz}(x)\pi_{\varphi_\bz}(\tau^n(y))\pi_{\varphi_\bz}(z))\\
&=E_\tau(\pi_{\varphi_\bz}(x)E_\tau(\pi_{\varphi_\bz}(y)) E_\tau(\pi_{\varphi_\bz}(z))\,,
\end{align*}
and we are done. In the exchangeable case, the block-singleton condition can be proved analogously.\\
The full-product condition is more laborious to prove.\\
Consider a separating family ({\it e.g.} the set of all normal states) of normal states of $\mathcal{T}_{\varphi_\bz}$, say $\{\om_i^0: i\in I\}$ along with their normal spreadable extension to the whole $\mathcal{R}_{\varphi_\bz}$, $\{\widetilde{\om}_i: i\in I\}$. For each $i$, we denote by $\om_i$ the corresponding
state at the $C^*$-algebra level, $\om_i(x)=\widetilde{\om}_i (\pi_{\varphi_\bz}(x))$, $x$ in $\ast_\bn\ga$.\\
Let $\Psi$ be the linear map from $\mathcal{R}_{\varphi_\bz}$  to $\cb(\oplus_i\ch_{{{\om}_i}})$ given by $\Psi(S):=\oplus_{i\in I} E_{\om_i} \pi_{\widetilde{\om}_i}(S) E_{\om_i}$, where $E_{\om_i}$ is the orthogonal projection onto the closed subspace $\ch_{{{\om}_i}}^\tau\subset \ch_{{{\om}_i}}= \ch_{\widetilde{\om}_i}$. Note that when $\Psi$ is restricted to $\ct_{\varphi_\bz}$, by Theorem \ref{HSspreadonL2}, it is precisely the direct sum of the GNS representations corresponding to $\{\om_i^0: i\in I\}$  and thus is faithful on $\ct_{\varphi_\bz}$. \\
We claim that $\Psi(ST)=\Psi(S)\Psi(T)$ for all 
$S, T$ in $\mathcal{R}_{\varphi_\bz}$ with disjoint supports and that $\Psi (T)=\Psi(E_\tau(T))$ for all
$T$ in $\mathcal{R}_{\varphi_\bz}$. 
Using these claims, when $S,T\in \mathcal{R}_{\varphi_\bz}$ have disjoint supports we have 
\begin{align*}
\Psi(E_\tau(ST))&=\Psi(ST)=\Psi(S)\Psi(T)=\Psi(E_\tau(S))\Psi(E_\tau(T))\\
&=\Psi(E_\tau(S)E_\tau(T))\,,
\end{align*}
by the fact that the restriction of $\Psi$ to $\ct_{\varphi_{\bz}}$ is multiplicative since it is a representation.
Hence $E_\tau(ST)=E_\tau(S)E_\tau(T)$ by injectivity of $\Psi$ restricted to the tail algebra.\\
It remains to prove the claims. For the first one, note that  if $T, S$ are localized elements of
$\mathcal{R}_{\varphi_\bz}$ with disjoint supports, we have $E_{{\om}_i} \pi_{\widetilde{\om}_i}(TS) E_{{\om}_i}=E_{{\om}_i} \pi_{\widetilde{\om}_i}(T)E_{{\om}_i} \pi_{\widetilde{\om}_i}(S) E_{{\om}_i}$ by virtue of
Lemma \ref{insertproj}.\\
For the second claim, by density, it is enough to verify the claimed equality on elements of the type
$\pi_{\varphi_\bz}(x)$, $x$ in $\ast_\bn\ga$, only. But in this case we simply  have
\begin{align*}
E_{{\om}_i}\pi_{\widetilde{\om}_i}(E_\tau(\pi_{\varphi_\bz}(x))) E_{{\om}_i}&=\lim_n E_{{\om}_i} \pi_{\widetilde{\om}_i}(\tau^n(x)) E_{{\om}_i}=\lim_n
E_{{\om}_i} V_\tau^n \pi_{\widetilde{\om}_i}(x) E_{{\om}_i}\\
&= E_{{\om}_i}\pi_{\widetilde{\om}_i}(x)E_{{\om}_i}\, ,
\end{align*}
where $V_\tau$ is the isometric implementer of the shift in the representation $\pi_{\widetilde{\om}_i}$.
\end{proof}

\begin{rem}
At this point, one might wonder whether Theorem \ref{deFinetti} has a converse.
In this regard, one should look at the nice example  by Gohm and K\"{o}stler, in \cite[Theorem 5.6]{GK08}, of a stationary quantum stochastic process which fails to be spreadable but whose normal, shift-invariant and $\varphi$-invariant conditional expectation  nevertheless 
satisfies the full product condition. Now, it is a matter of simple calculations to ascertain that the conditional expectation of that process enjoys the forward-decoupling property as well.
This demonstrates that spreadability cannot be recovered from assuming stationarity of the process and forward-decoupling property of its (unique normal, shift-invariant, $\varphi$-invariant) conditional expectation, together with the full product condition, which prevents our Theorem \ref{deFinetti}
from having a converse.
\end{rem}
Going back to the proofs of the results we obtained, one can realize that most of them do not take full advantage of the fact that our processes are spreadable. In fact, what is really made use of is a weaker 
distributional symmetry that deserves to be singled out in a definition.
\begin{defin}\label{weakspread}
A quantum stochastic process with distribution $\varphi$ on $\ast_\bn\ga$ is said to be \emph{weakly spreadable} if
$$\varphi(xyx)= \varphi(x\tau(y)z)$$
 for all localized $x, y, z$ in $\ast_\bn\ga$ with supports contained in finite subsets $F_x, F_y, F_z$ respectively  with $F_y > F_x \cup F_z$.
\end{defin}

\begin{rem}
Apart from those involving the use of Proposition \ref{insertproj}, all the  main results of the present paper continue to hold
with spreadability replaced by weak spreadability. In particular, the non-commutative Olshen theorem, Theorem \ref{HSonZ}, the existence of the conditional expectation onto the tail algebra, Theorem \ref{condexpgen}, and the non-commutative Hewitt-Savage 0--1 law, Theorem \ref{01lawspread},  can be stated replacing spreadability with
weak spreadability. The hierarchy of ergodicities, Theorem \ref{ergohierspread}, continues to work as well for weakly spreadable states as long as  the full product condition is omitted from the statement in (ii).
Moreover, the conditional expectation built out of a weakly spreadable state continues to enjoy the forward-decoupling condition.\\
\end{rem}
We are now ready to state  the full non-commutative de Finetti theorem alluded to at the beginning of the section.  To do so, let us establish a bit of language: we say that the tail algebra of a process is expected if
there exists a normal conditional expectation from $\mathcal{R}_{\varphi_\bz}$ onto $\mathcal{T}_{\varphi_\bz}$ (which is $*$-isomorphic with $\mathcal{T}_\varphi$
via the restriction map in the weakly spreadable case as well).

\begin{thm}\label{deFinetticonv}
For a stationary quantum stochastic process with distribution $\varphi\in\mathcal{S}(\ast_\bn\ga)$, the following are equivalent:
\begin{itemize}
\item [(i)] the process is weakly spreadable;
\item [(ii)] the tail algebra $\mathcal{T}_\varphi$  is expected with
a unique shift-invariant normal, $\varphi$-invariant conditional expectation, which is also forward-decoupling 
\end{itemize}
\end{thm}

\begin{proof}
The implication (i)$\Rightarrow$(ii) is a consequence of Theorems \ref{condexpgen}, which holds  for weakly spreadable processes as well, and
the part of Theorem \ref{deFinetti} concerned with the forward-decoupling condition for the conditional expectation, which can be derived in the same was as we did in the spreadable case.\\
The implication (ii)$\Rightarrow$(i) follows from the computation below. If $\widetilde{\varphi}$ is the vector state on $\cb(\ch_{\varphi_\bz})$ associated with the GNS vector $\xi_{\varphi_\bz}$, and $x, y, z$ are localized elements of $\ast_\bn\ga$ such that the support
of $y$ lies to the right of that of both $x$ and $z$, we have:
\begin{align*}
\varphi(xyz)&=\widetilde{\varphi}(E_\tau(\pi_{\varphi_\bz}(xyz)))=\widetilde{\varphi}(E_\tau(\pi_{\varphi_\bz}(x)E_\tau(\pi_{\varphi_\bz}(y))\pi_{\varphi_\bz}(z)))\\
&= \widetilde{\varphi}(E_\tau(\pi_{\varphi_\bz}(x)E_\tau(\tau(\pi_{\varphi_\bz}(y)))\pi_{\varphi_\bz}(z)))\\
&=\widetilde{\varphi} (E_\tau(\pi_{\varphi_\bz}(x)\tau(\pi_{\varphi_\bz}(y))\pi_{\varphi_\bz}(z)))=\varphi(x\tau(y)z)\,,
\end{align*}
and we are done.
\end{proof}

Interestingly, if a weakly spreadable state  factorizes in the sense of Definition \ref{distri} through a quotient  by a $\bj_\bn$-invariant ideal $I$  which is linearly generated by ordered monomials (that is 
the set of all equivalence classes of the type $[\iota_{j_1}(a_1)\ldots \iota_{j_n} (a_n)]$, with $j_1<\ldots< j_n$ or $j_1>\ldots>j_n$ and  $a_1, \ldots, a_n$ in $\ga$, generates $\ast_\bn\ga/ I$  as a Banach space),
then the state is spreadable. \\
What is more, if the quotient is by a $\bp_\bn$-invariant ideal, then exchangeability is the same as spreadability, meaning that exchangeability can also be characterized in terms of stationarity
and the forward-decoupling property of the conditional expectation.

\begin{cor}\label{deFinettiquot}
For a stationary quantum stochastic process whose distribution $\varphi$ factorizes through a quotient of
$\ast_\bn\ga$ by an ideal $I$ such that
\begin{itemize}
\item [(a)] $I$ is $\bj_\bn$-invariant (respectively $\bp_\bn$-invariant)
\item [(b)] the quotient $\ast_\bn\ga/I$  is linearly generated by ordered monomials,

\end{itemize}

 the following are equivalent:
\begin{itemize}
\item [(i)] $\varphi$ is spreadable (respectively exchangeable);
\item [(ii)] $\varphi$ is weakly spreadable;
\item [(iii)] the tail algebra $\mathcal{T}_\varphi$ of $\varphi$ is  expected with
a unique shift-invariant normal, $\varphi$-invariant conditional expectation, which is also forward-decoupling;
\item [(iv)] the tail algebra $\mathcal{T}_\varphi$ of $\varphi$ is expected 
with
a unique shift-invariant normal, $\varphi$-invariant conditional expectation, which is also forward-decoupling and satisfies
the full product condition.
\end{itemize}
\end{cor}

\begin{proof}
We need only prove the implication (ii)$\Rightarrow$ (i). This can be seen in the following way. 
Since $\varphi$ factorizes through a quotient as in the statement, it is enough
to check spreadability only on words of the type $\iota_{i_1}(a_1)\cdots i_{j_k}(a_k) i_{j_k}(a_k) i_{j_{k+1}}(a_{k+1})\cdots\iota_{i_n}(a_n)$ with ordered indices. In this case,
Definition \ref{weakspread} simply yields the following invariance property for $\varphi$:
\begin{align*}
&\varphi(\iota_{i_1}(a_1)\cdots i_{j_k}(a_k) i_{j_k}(a_k) i_{j_{k+1}}(a_{k+1})\cdots\iota_{i_n}(a_n))=\\
&\varphi(i_{j_1}(a_1)\cdots i_{j_k}(a_k) i_{j_{k+1}+1}(a_{k+1})\cdots i_{j_n+1}(a_n))\,,
\end{align*}
which is the invariance of $\varphi$ under the partial shifts $\theta_{j_{k+1}}$, and these generate the semigroup $\bj_\bn$.\\
Finally, the exchangeable case follows from \cite[Theorem 3.2]{ADR}.
\end{proof}
In particular, Corollary \ref{deFinettiquot} applies to tensor products, covering the classical setting of random variables taking their values on a compact Hausdorff space, and, more generally, to twisted tensor products, which are described and dealt with in Section \ref{models}.
The case of unbounded random variables is finally covered in the next section, where non-unital processes are  considered.

\begin{rem}
It remains open whether weak spredability automatically implies the full product condition for the expectation onto the tail algebra.
\end{rem}
\section{Non-unital processes}\label{nunital}

By a non-unital quantum stochastic process we mean a quadruple $(\gb, \om, \ga, \{\iota_j\}_{ j\in\bn})$, where now, as opposed to the unital case, the sample algebra $\ga$ is allowed to be non-unital and the $*$-homomorphisms $\iota_j$, too, can be non-unital.
Furthermore, we require that, for every $j$ in $\bn$, the positive functional $\om\circ\iota_j$ on $\ga$ is normalized.
Note that this requirement is automatically satisfied in the unital case.\\
The distribution of such a process is the
 normalized positive functional $\varphi: *_\bn^{0}\ga \to \bc$ on the non-unital free product $\ast_\bn^{0}\ga$ defined by 
\begin{align}
&\varphi (i_{j_1}(a_1)i_{j_2}(a_2)\cdots i_{j_n}(a_n)):=\om(\iota_{j_1}(a_1)\iota_{j_2}(a_2)\cdots \iota_{j_n}(a_n))
\end{align}
for all $n\in\bn$, $j_1\neq j_2\neq \cdots \neq j_n\in\bn$, 
$a_1$, \ldots, $a_n\in\ga$, where $i_k:\ga\rightarrow\ast_\bn^{0}\ga$ is the $k$-th embedding of $\ga$ into its infinite non-unital free product.
As in the unital case, $\varphi$ is a well-defined positive functional thanks to the universal property enjoyed by $\ast_\bn^{0}\ga$, {\it i.e.} for any family
$\{f_k\}_{k\in\bn}$ of possibly non-unital $*$-homomorphisms $f_k:\ga\rightarrow\gb$, with $\gb$ being a possibly non-unital $C^*$-algebra, there exists a unique
$*$-homomorphism  $\Phi: \ast_\bn^{0}\ga\rightarrow\gb$ such that $\Phi(i_k(a))=f_k(a)$, $a$ in $\ga$ and $k$ in $\bn$.\\
For any $C^*$-algebra $\gb$, we denote by $\widetilde{\gb}=\gb\oplus\mathbb{C}$ the one-point unitalization of $\gb$, with the tacit understanding that if $\gb$ is already unital, we add the unit all the same.\\
We denote by $\cs(\ast_\bn^0\ga)$  the convex compact set of all positive functionals on $\ast_\bn^0\ga$ with norm less than or equal to $1$, the zero functional included. Note that $\cs(\ast_\bn^0\ga)$ is affinely homeomorphic with $\cs(\widetilde{\ast_\bn^0\ga})$,  the compact convex set of all states of
$\widetilde{\ast_\bn^0\ga}$, which is the one-point unitalization of $\ast_\bn^0\ga$, via $\cs(\ast_\bn^0\ga)\ni\varphi\mapsto\widetilde{\varphi}\in\cs(\widetilde{\ast_\bn^0\ga})$, with $\widetilde{\varphi}(x+\lambda I)=\varphi(x)+\lambda$, $x$ in $\ast_\bn^0\ga$ and $\lambda$ in $\bc$.\\
Now, by the universal property  of the unital free product $\ast_\bn\widetilde{\ga}$, there exists a unital $*$-homomorphism
 $\Psi: \ast_\bn\widetilde{\ga}\rightarrow \widetilde{\ast_\bn^0\ga}$ uniquely determined by
$\Psi(i_k(a+\lambda I))=i_k'(a)+\lambda I$ for all $a$ in $\ga$ and $k$ in $\bn$. It is easy to see that
$\Psi$ is a $*$-isomorphism, as can be checked by constructing its inverse by means of the universal property of the non-unital
free product $\ast_\bn^0\ga$. In particular, the two compact convex sets $\cs({\ast_\bn^0\ga)}$ and $\cs(\ast_\bn\widetilde{\ga})$ can be identified.\\
That said, we can move on to give the definition of tail algebra for non-unital processes.

\begin{defin}
The tail algebra $\mathcal{T}_\varphi$ of a non-unital stochastic process with distribution $\varphi$ 
 is the von Neumann algebra
$$\mathcal{T}_\varphi=\bigcap_{n\in\bn}  \overline{\{\pi_\varphi(i_k(\ga)): k\geq n\}}^w\subset\cb(\ch_\varphi)\, ,$$ 
where $\overline{\mathcal{A}}^w$ denotes the closure of a $*$-subalgebra $\mathcal{A}$ in any of the weak, strong, ultraweak, ultrastrong operator topologies.
\end{defin}

\begin{rem}
Out of any non-unital process $(\gb, \om, \ga, \{\iota_j\}_{ j\in\bn})$ one can obtain a unital process
 $(\widetilde{\gb}, \widetilde{\om,} \widetilde{\ga}, \{\widetilde{\iota_j}\}_{ j\in\bn})$, where, for all $j$ in $\bn$,
$\widetilde{\iota_j}: \widetilde{\ga}\rightarrow\widetilde{\gb}$ is the unital extension of
$\iota_j$, {\it i.e.} $\widetilde{\iota_j}(a+\lambda I)=\iota_j(a)+\lambda I$, $a$ in $\ga$ and $\lambda$ in $\bc$.
If $\varphi$ is the distribution of the original process $(\gb, \om, \ga, \{\iota_j\}_{ j\in\bn})$, then $\widetilde{\varphi}$ is the distribution 
of the unitalized process $(\widetilde{\gb}, \widetilde{\om,} \widetilde{\ga}, \{\widetilde{\iota_j}\}_{ j\in\bn})$, as is easily verified.\\
As in the unital case, we still denote by $\mathcal{R}_\varphi$ the (possibly degenerate) weakly closed $*$-algebra generated by the GNS representation of $\varphi$, that is $\overline{\pi_\varphi(\ast_\bn^0\ga)}^w$ and by $\mathcal{R}_{\varphi_\bz}$ the weakly closed $*$-algebra $\overline{  \pi_{\widetilde{\varphi}_\bz} (\ast_\bn^0\ga) }^w $.\\
Note that the Olshen Theorem carries over to the non-unital case with the same proof as in Theorem \ref{HSonZ}.\\
Furthermore, thanks to the unitalization precedure, for any non-unital quantum process $(\gb, \om, \ga, \{\iota_j\}_{ j\in\bn})$  there still exists a unique normal invariant conditional expectation $E_\tau$ onto the tail algebra $\mathcal{T}_\varphi$, which is obtained by restricting to
$\mathcal{R}_{\varphi_\bz}$ the conditional expectation $\widetilde{E_\tau}$ of the unitalized process.
Note that in the non-unital case, it is not necessarily true that $\bc I$ is contained in $\mathcal{T}_\varphi$ due to non-unitality of the $i_j$'s. Nevertheless, $\mathcal{T}_\varphi$ can never be $0$ unless $\varphi=0$.
\end{rem}
In the next definition we collect the analogues for non-unital processes of the equivalent properties in the level (iii) of the hierarchy of ergodicities, that is those of the Hewitt-Savage $0$--$1$ law.

\begin{defin}\label{nonunital01}
A stationary positive linear functional  $\varphi$ in $\cs(\ast_\bn^0\ga)$:
\begin{itemize} 
\item [(i)] is forward-decoupling  if $\varphi(xyz)=0$ when $\varphi(y)=0$, provided that $x,y,z$ are localized with supports contained in finite subsets $F_x, F_y, F_z$ respectively,  with $F_y > F_x \cup F_z$.

\item [(ii)] is $\pi$-attractive (with $\pi$ being a representation of $\ast^0_\bn\ga$) if for all $\om$ in
$\cs_\pi$ one has $\lim_n \om(\tau^n(x))=0$ if $\varphi(x)=0$, $x$ in $\ast^0_\bn\ga$.
\end{itemize}
\end{defin}

\begin{rem}
As will be clear later on, see \emph{e.g.} Example \ref{boolean}, it is not necessarily true that if $\varphi$ satisfisies any of (i)-(ii) then its extension $\widetilde{\varphi}$ to
$\ast_\bn\widetilde{\ga}$ satisfies the corresponding properties for states in the unital free product.
Of course, if the process is unital the above Definitions agree with the respective ones for unital processes in Definition \ref{equi}.
\end{rem}


The next result is the hierarchy of ergodicities for the distributions of weakly spreadable non-unital processes, which can be defined in the exact same way as in the unital case.

\begin{thm}\label{mainnonunital}
Let $\varphi$ in $\cs(\ast_\bn^0\ga)$ be the distribution of a weakly spreadable non-unital quantum stochastic process, then:\\
\begin{itemize}
\item [(i)]    $\varphi$ is extreme in $\cs^{\tau}(\ast_\bn^0\ga)$ $\Leftrightarrow$ $\{\pi_\varphi (\ast_\bn^0\ga), V_\tau  \}'=\mathbb{C}$;
\item [(ii)] ${\rm dim}\,E_\varphi=1$ $\Leftrightarrow$ $\varphi$ is strongly clustering $\Leftrightarrow$ $\varphi$ is an ordered product state;
\item [(iii)] $\mathcal{T}_\varphi=\bc F$ for some projection $F$ $\Leftrightarrow$ $\varphi$ is forward-decoupling $\Leftrightarrow$ $\varphi$ is $\pi_{\varphi_\bz}$-attractive;
\item [(iv)]   $\varphi$ is extreme in $\cs^{\tau}(\ast^0_\bn\ga)$ and has central support ($\xi_\varphi$ is separating for
$\pi_\varphi(\ast_\bn^0\ga)''$).
\end{itemize}
Furthermore, (iv) $\Rightarrow$ (iii) $\Rightarrow$ (ii) $\Rightarrow$ (i).

\end{thm}

\begin{proof}

 Since the proof of this theorem is virtually the same as that of Theorem \ref{ergohierspread} by passing to the unitalized process,
we will just comment on those points where extra care has to be taken.\\
The equivalences in (iii), which represent the Hewitt-Savage $0$--$1$ law for non-unital processes can be done as in the proof of the analogous implications in Theorem \ref{01lawspread}. 
However, for the implication that the future-decoupling condition gives a trivial tail algebra we need to ascertain that
${\rm dim\,}E_\varphi=1$, which is used in the proof of Theorem \ref{01lawspread}.  We will prove  this by showing that the extension
$\widetilde{\varphi}$ is an ordered product state, to which one can then apply Proposition \ref{ergospread}. This amounts to showing that
$\widetilde{\varphi}(xy)=0$ if $\widetilde{\varphi}(y)=0$ for $x, y$ in $\widetilde{\ast_\bn^0\ga}$ localized and with the support of $y$ to the right of the support of $x$. 
Let $\cb\subset\ast_\bn^0\ga$ the $C^*$-subalgebra generated by those localized elements whose support is the same as that of $y=y'+\beta I$ for some
$y$ in $\cb$ and $\beta$ in $\bc$ (note that $(\varphi(y')=-\beta$ as follows from $\widetilde{\varphi}(y)=0$).
Let $\{e_\lambda: \lambda\in\Lambda\}$ be an approximate unit of $\gb$. Since the restriction
of $\varphi$ to $\gb$ has norm $1$ by hypothesis, we can apply Lemma \ref{approxunit} and conclude
$\lim_\l\pi_\varphi(e_\l)\xi_\varphi=\xi_\varphi$. \\
Now, writing $x=x'+\alpha I$, $x'$ in $\ast_\bn^0\ga$ and $\a$ in $\bc$, and taking into account that $\varphi(e_\l)$ tends to $1$, we find:
\begin{align*}
&\widetilde{\varphi}(xy)= \widetilde{\varphi}\big(x(y'+\beta I)\big)=\lim_\l \widetilde{\varphi}\big(x(y'+\frac{1}{\varphi(e_\l)}\beta e_\l)\big)=\\
&\lim_\l \varphi\left(x'\big(y'+\frac{1}{\varphi(e_\l)}\beta e_\l\big)\right)+ \alpha\varphi\left(y'+\frac{1}{\varphi(e_\l)}\beta e_\l\right)=0
\end{align*}
since $\varphi\big(y'+\frac{1}{\varphi(e_\l)}\beta e_\l\big)=\varphi(y')+\beta=0$ for every $\lambda$ and
the support of $y'+\frac{1}{\varphi(e_\l)}\beta e_\l$ lies to the right of the support of $x$ by construction.\\
The implication that a trivial tail algebra yields  attractivity, too, can be done as in the proof of Theorem  \ref{01lawspread}
as long as one first makes sure that $E_\tau(\pi_{\widetilde{\varphi}_\bz}(x))=\varphi(x)F$, $x$ in $\ast_\bn^0\ga$, where $F$ is the projection such that
$\mathcal{T}_\varphi=\bc F$, which is what we do below.\\
If $\mathcal{T}_\varphi=\bc F$, then $E_\tau(T)= \lambda(T)F$, $T$ in $\mathcal{R}_{\varphi_\bz}$,  for a suitable linear functional $\lambda$ on $\mathcal{R}_{\varphi_\bz}$. We first prove that $\lambda$ is normalized. To this end, note that $\|E_\tau\|=1$ (which is certainly true as $E_\tau(F)=F$) implies $\|\lambda\|=1$ because of the equality $|\lambda(T)|= \|\lambda(T)F\|=\|E_\tau(T)\|$.\\
Now by $\varphi$-invariance of $E_\tau$ we have
$|\lambda(T)|\varphi(F)=|\varphi(T)|$, for all $T$ in $\mathcal{R}_{\varphi_\bz}$, where with abuse of notation we have denoted by $\varphi$ its normal extension on $\mathcal{R}_{\varphi_\bz}$.
By taking the sup for $T$ running in the unit ball of $\mathcal{R}_{\varphi_\bz}$ on both the right and left side of the previous equality
we get $\varphi(F)=1$, hence, again by $\varphi$-invariance of $E_\tau$,
$\lambda=\varphi$.\\

\end{proof}

Having at hand a theory for non-unital processes allows us to tackle  monotone/antimonotone and  Boolean independence as well.

\begin{example}\label{boolean}
Following \cite{HO}, a non-unital process is said to be Boolean independent, identically distributed if its distribution is identically distributed   and satisfies:
\item $$\varphi(\iota_{i_1}(a_1) \cdots \iota_{i_m}(a_m))=\varphi(\iota_{i_1}(a_1))\varphi(\iota_{i_2}(a_2) \cdots \iota_{i_m}(a_m))$$
for  \(i_1 \neq i_2 \neq \cdots \neq i_n\).\\
It  is easy to see that a Boolean independent, identically distributed process is exchangeable and  satisfies the forward-decoupling condition (ii) in Definition \ref{nonunital01}. Therefore, by Theorem \ref{mainnonunital} it gives rise to a one-dimensional
tail algebra. \\
Moreover,  the projection $F$ such that $\mathcal{T}_\varphi=\bc F$ will in general be proper, as shown in \cite[Example 7.3]{Liu}.
This example also shows that the condition for a normalized positive functional to be forward-decoupling at the level of the non-unital process does not imply that the distribution of the unitalized process is still forward-decoupling in the sense of Definition \ref{equi}.\\
Importantly, when the sample algebra admits a single self-adjoint generator, that is $\ga=C(K)$ for some compact subset $K\subset\br$ of the real line, by using the characterization of Boolean  conditional i.i.d. in terms of invariance under the linear coaction of the quantum semigroup $\mathcal{B}_s(n)$ given in Theorem 1.1 of \cite{Liu}, to which the reader is referred for all the needed definitions,
we can come to the following characterization of Boolean i.i.d. processes: a quantum stochastic process 
is Boolean i.i.d. if and only if its distribution is 
$\mathcal{B}_s(n)$-invariant for all $n$ in $\bn$ and satisfies any of the three equivalent conditions of the non-commutative Hewitt-Savage $0$--$1$ law, Theorem
\ref{01lawspread}.
\end{example}

\begin{example}\label{monotone}
Following \cite{HO}, a (non-unital) process with distribution $\varphi$ and sample algebra $\ga$ is said to be
\begin{itemize}
\item[-] monotone independent if
$$\varphi(i_{j_1}(a_1) \cdots i_{j_m}(a_m))=\varphi(i_{j_p}(a_p))\varphi(i_{j_1}(a_1) \cdots i_{j_{p-1}}(a_{p-1}) i_{j_{p+1}}(a_{p+1}) \cdots i_{j_m}(a_m))$$   
when  \(j_1 \neq j_2 \neq \cdots \neq j_n\) and 
$j_p>j_{p\pm 1}$;
\item[-] anti-monotone independent if
$$\varphi(i_{j_1}(a_1) \cdots i_{j_m}(a_m))=\varphi(i_{j_p}(a_p))\varphi(i_{j_1}(a_1) \cdots i_{j_{p-1}}(a_{p-1}) i_{j_{p+1}}(a_{p+1}) \cdots i_{j_m}(a_m))$$   
when  \(j_1 \neq j_2 \neq \cdots \neq j_n\) and 
$j_p<j_{p\pm 1}$.\\
 \end{itemize}
It is easy  to see that a monotone/antimonotone identically distributed process is spreadable and satisfies the forward-decoupling condition.
In particular, by Theorem \ref{mainnonunital} the tail algebra of a monotone/antimonotone identically distributed process is one-dimensional. In the monotone case, the unitalized process continues to satisfy the forward-decoupling condition in Definition \ref{equi}, meaning that the the projection $F$ such that $\mathcal{T}_\varphi=\bc F$ coincides with the identity $I$ of $\cb(\ch_\varphi)$.\\
Unlike the monotone case, the tail algebra of an antimonotone independent , identically distributed process can be degenerate, as is shown in Example \ref{monotonetail}.\\
As with Boolean independence, for processes with a sample algebra admiting a single self-adjoint generator, that is $\ga=C(K)$ for some compact subset $K\subset\br$ of the real line, by using the characterization of monotone  conditional i.i.d. in terms of invariance under the coaction of the quantum semigroups $M(n,k)$ of quantum spreadability given in Theorem 1.1 of \cite{LiuBM}, to which the reader is referred for all the needed definitions,
we can come to the following characterization of monotone i.i.d. processes: a quantum stochastic process 
is monotone i.i.d. if and only if its distribution is 
$M(n, k)$-invariant for all $n$ in $\bn$ and $k<n$, and satisfies any of the three equivalent conditions of the non-commutative Hewitt-Savage $0$--$1$ law, Theorem
\ref{01lawspread}. The antimonotone case follows from the monotone case by reflection, that is, by reversing the order of the variables.

\end{example}

\begin{example}\label{monotonetail}
In this example of a non-unital quantum stochastic process the sample algebra $\ga$ is given by the $2$ by $2$ matrices $\mathbb{M}_2(\bc)$. In order to define all of the data of the quantum stochastic process, we need to recall what the antimonotone Fock space looks like.\\
 For $k\geq 1$, set $I_k:=\{(i_1,i_2,\ldots,i_k) \mid i_1< i_2 < \cdots <i_k, i_j\in \mathbb{N}\}$. The  antimonotone Fock space is the Hilbert space $$\cf_{{\rm amon}}(\ch):=\bigoplus_{k=0}^{\infty} \ch_k\,,$$ 
where $\ch_0=\mathbb{C}$ and for every $k\geq 1$ $\ch_k:=\ell^2(I_k)$. Henceforth we will denote by $\Om$ the vector in $\cf_{{\rm amon}}(\ch)$
given by $(1, 0, 0, \ldots)$ and refer to it  as the Fock vacuum.\\
The canonical basis of the antimonotone Fock space is obtained in the following way. If $(i_1,i_2,\ldots,i_k)\in I_k$ is an increasing sequence  of integers, we denote by $e_{(i_1,i_2,\ldots,i_k)}\in \ell^2(I_k)$ the square summable sequence that is always zero but at $(i_1, i_2, \ldots, i_k)$, where
it is $1$. The vector corresponding to the empty set
is just the vacuum $\Om$, and the corresponding vector state $\om_\Om$ is called the vacuum state. 
For every $i\in \mathbb{N}$, the monotone creation and annihilation operators are respectively given  by $a^\dag_i\Om=e_{(i)}$, $a_i\Om=0$ and
\begin{equation*}
a^\dagger_i e_{(i_1,i_2,\ldots,i_k)}:=\left\{
\begin{array}{ll}
e_{(i,i_1,i_2,\ldots,i_k)} & \text{if}\,\, i< i_1 \\
0 & \text{otherwise}, \\
\end{array}
\right.
\end{equation*}
\begin{equation*}
a_ie_{(i_1,i_2,\ldots,i_k)}:=\left\{
\begin{array}{ll}
e_{(i_2,\ldots,i_k)} & \text{if}\,\, k\geq 1\,\,\,\,\,\, \text{and}\,\,\,\,\,\, i=i_1\\
0 & \text{otherwise}. \\
\end{array}
\right.
\end{equation*}
%
The $C^*$-probability space $(\gb, \om)$  of the process is obtained by taking $\gb=\cb(\cf_{{\rm amon}}(\ch))$, and choosing $\om$ to be the vacuum state $\om_\Omega$.\\
Lastly, the $*$-homomorphisms $\iota_j: \mathbb{M}_2(\bc)\rightarrow\gb$ are given by $\iota_j(A)=a_j$, where
$A:=
 \left(
\begin{array}{ll}
0 & 1 \\
0 & 0
\end{array}
\right)\in\bm_2(\bc)
$.
The non-unital quantum process thus obtained is well known to be antimonotone independent and identically distributed.\\
Note that by considering the Hilbert space with reversed order relations between indices, one obtains a monotone independent process.\\
We claim
that the tail algebra of the antimonotone process is $\bc P_\Omega$, where $P_\Omega$ is the orthogonal projection onto the vacuum vector.\\
Indeed, it is easy to see  that $\lim_n a_na_n^\dagger= P_\Omega$ in the strong operator topology, which shows that $P_\Omega$ must lie in the tail algebra. On the other hand, the tail algebra must be one-dimensional, as explained in Example 
\ref{monotone}, hence the conclusion.\\
This process also provides a good illustration of the fact that the past tail algebra, see Remark \ref{pasttail}, and the tail algebra in general are not $*$-isomorphic. To see this, we consider the unitalized process obtained out of the above antimonotone process. Its tail algebra is not one-dimensional, containing $\bc P_\Omega\oplus \bc (I-P_\Omega)$. On the other hand, the past tail algebra of the unitalized process is seen to be $*$-isomorphic with the tail algebra of the unitalized monotone process,
which has a trivial tail algebra  $\mathbb{C}I$ by Example \ref{monotone}.\\
As a final remark, we would like to point out that the unitalized monotone process fails to satisfy the block-singleton condition while satisfying the forward-decoupling condition, for we have $0=\om_\Omega (a_1a_0 a_0^\dagger a_1^\dagger)\neq\om_\Omega(a_1a_1^\dagger)\om_\Omega(a_0a_0^\dagger)=1$.

\end{example}

\begin{rem}
The general de Finetti theorem, Theorem \ref{deFinetticonv}, continues to hold for non-unital processes as well with the same proof. Indeed, it is enough to take into account the unitalized process and consider the restriction of the conditional expecation to the relative non-unital algebras. In particular, the classical case of possibly unbounded real-valued random variables now fits in the theory developed in the paper.
\end{rem}

\section{Models}\label{models}

\subsection{Twisted tensor products}
In this section we provide a thorough account of the examples we alluded to  throughout the paper, especially in relation to the strictness of the implication (iii)$\Rightarrow$(ii) in the hierarchy of ergodicities
established in Theorem \ref{ergohierspread}.  Twisted tensor products by a bicharacter of the dual of a grading group $G$
give rise to quantum stochastic processes whose distributions factorize as  full product states but whose
tail algebra may fail nonetheless to be trival because it actually
keeps track of  the $G$-orbit of the state itself.\\
Twisted tensor products are particularly workable examples of quotients of the infinite free product, which are useful in that
the ideal out of which they are obtained is always invariant under spreadable maps and may be invariant under permutations depending on the bicharacter.\\
Suppose that the sample algebra $\ga$ is given more structure. namely  that it is a $G$-graded $C^*$-algebra, with $G$ being an abelian compact group. This means that we actually have a
  C$^*$-dynamical system \( (\ga, G, \gamma) \), where \( \mathfrak{A} \) is a unital \( C^* \)-algebra, and \( \gamma : G \to \text{Aut}(\ga) \) a group homomorphism, pointwise norm continuous, {\it i.e.}, for every \( a \) in \( \ga \), the function \( G \ni g \mapsto \gamma_g(a) \in \mathfrak{A} \) is continuous w.r.t the norm topology. 
 The eigenspace of the action associated with \( \chi \) in \( \widehat{G} \) ($\widehat{G}$ is the Pontryagin dual of $G$) is the subspace 
\[V_\chi := \{a \in \ga: \gamma_g(a) = \chi(g)a, \text{ for all } g \in G\}\subset\ga.\]
The elements of each $V_\chi$ are referred to as being homogeneous.
For every \( \chi, \eta \) in \( \widehat{G} \), one has \( V_\chi V_\eta \subset V_{\chi\eta} \) and \( V_\chi^* = V_{\chi^{-1}} \).
The algebraic direct sum \( \ga_0:= \bigoplus_{\chi \in \hat{G}} V_\chi \) is a dense \( * \)-algebra of \( \ga \), see \emph{e.g.} \cite[Theorem 2.5]{Salle16}, and it is 
known as the algebraic layer of \( \ga \).  If $a$ in $\ga$ is homogeneous with $\g_g(a)=\chi(g)a$ for all $g$ we say that 
$\chi$ is the degree of $a$ and denote it by $\deg a$.\\
To any bicharacter $v: \widehat{G}\times\widehat{G}\rightarrow\bt$ (a bicharacter is a function that is a character w.r.t. each separate variable), it is possible to associate a quotient of $\ast_\bn\ga$ by the $\bj_\bn$-invariant ideal
$I_v$ generated by the set
 $$\{\iota_k(a)\iota_l(b) -v(\deg a, \deg b)\iota_l(b)\iota_k(a): a, b\, \textrm{homogeneous in}\,\ga, k< l\, \in \bn\}\,.$$
The $C^*$-algebra thus obtained is known as the infinite twisted tensor product w.r.t. to $v$ (note that the case
$v=1$ gives back the usual maximal tensor product) and is denoted by $\otimes_v^\bn\ga$. 
By construction the action of $\bj_\bn$ on the infinite free product passes to its quotient $\otimes_v^\bn\ga$.
The action of $\bp_\bn$, instead, passes to $\otimes_v^\bn\ga$ if and only if the bicharacter $v$ is antisymmetric
($v(\chi, \eta)= \overline{v(\eta, \chi)}$ for all $\chi, \eta$ in $\widehat{G}$).

A more concrete construction of the infinite  twisted tensor product is possible, see {\it e.g.} \cite{ADR, FidVic}.
Composing each $i_k$ with the quotient map $\ast_\bn\ga\mapsto\otimes_v^\bn\ga$
gives a unital embedding of $\ga$ into $\otimes_v^\bn\ga$ which by a minor abuse of notation
we continue to denote by $i_k$.\\
The infinite twisted tensor product itself is acted upon $G$ through the so-called gauge action, which is given by
$$\a_g(i_{j_1}(a_1)\cdots i_{j_k}(a_k))=i_{j_1}(\gamma_g(a_1))\cdots i_{j_k} (\gamma_g(a_k))$$
for $k, j_1, \ldots, j_k$ in $\bn$, and $a_1, \ldots, a_k$ in $\ga$.\\
Unlike usual tensor products, it is not always possible to consider the tensor product state of given states on the sample algebra $\ga$ when the bicharacter $v$ is not trivial. In fact, some compatibility between the assigned states and the bicharacter must be fulfilled. In order to recall the necessary and sufficient condition for the tensor
product state to exist, we recall that the spectral support of a state $\om$, as introduced in \cite{ADR}, is the set
\[
\mathrm{supp}_G \,\omega := \{\eta \in \widehat{G} : \exists a \in \ga_\eta \text{ with } \omega(a) \neq 0\}.
\]
Clearly, $\omega$ is $G$-invariant if and only if $\mathrm{supp}_G \,\omega$ is trivial. That said, we are ready to state
the announced result, \cite[Corollary 3.8]{ADR}.  
\begin{thm}\label{prodstates}
Let $\omega$ be a state on $\ga$. There exists a (unique) state $\times \omega$ on $\bigotimes_{v}^{\mathbb{N}} \ga$ such that
\[
\times \omega \big(\iota_{j_1}(a_1)\cdots \iota_{j_n}(a_n)\big) = \omega(a_1)\cdots \omega(a_n),
\]
for all $n \in \mathbb{N}$, $j_1 < \cdots < j_n$ in $\bn$, $a_1,\dots,a_n \in \ga$, if and only if for any $\eta_1, \eta_2 \in \mathrm{supp}_G \,\omega$ one has $v(\eta_1,\eta_2)=1$.
\end{thm}

In the sequel, tensor product states will be looked at  as distributions of quantum stochastic processes whose distribution factorizes through a twisted tensor product in the sense of Definition \ref{distri}.

\begin{rem}\label{tailGinv}
Infinite tensor product states as in Theorem \ref{prodstates} are all instances of full product states. Morever, it is easy to check that an infinite tensor product state
$\times \om$ is block-singleton if and only if $\om$ is $G$-invariant.
Also note that when the action of $G$ is not trivial, no infinite tensor product state satisfies any of the five natural non-commutative independence forms. 
\end{rem}

In the next section, we fix a particular specimen of infinite twisted tensor product to show what the tail algebra may look like when the assumption on the $G$-invariance of our product state is dropped.
As already mentioned,  the tail algebra will not be trivial so as to encode the $G$-orbit of the product state.

\subsection{The infinite non-commutative torus}
The infinite noncommutative torus $\mathbb{A}^{\bn}_{\alpha}$  with deformation parameter $\a\in\mathbb{R}$ is the    
 universal $C^* $-algebra generated by a countable set of unitaries $\{u_l : l \in \bn\}$ satisfying the commutation relations 
$ u_l u_k = e^{i2\pi\alpha} u_k u_l $ for any  $l, k \in \bn$  with $l < k$. \\
The $C^*$-algebra $\mathbb{A}^{\bn}_{\alpha}$ can also be obtained as an infinite twisted tensor product with sample
algebra $\ga= C(\bt)$ thought of as a $G$-graded $C^*$-algebra, with $G=\bt$ acting on $C(\bt)$ by rotation and bicharacter $v_\alpha:\bz\times\bz\rightarrow \bt$ given by $v_\alpha(l,k):=e^{i2\pi\alpha l k}$.\\
We recall that $\mathbb{A}^{\bn}_{\alpha}$ is endowed with a canonical trace, ${\rm tr}$, which is nothing but the infinite tensor
product of the state on $C(\bt)$ corresponding to the Haar  measure on $\bt$.\\
The convex set of all spreadable states on $\mathbb{A}^{\bn}_{\alpha}$, or equivalently by Corollary \ref{deFinettiquot} the set of all weakly spreadable states, is completely determined in \cite{CDGR}, where
the following has been proved.
\begin{thm}
The set of all spreadable states on $\mathbb{A}^{\bn}_{\alpha}$ is the Bauer simplex whose extreme points are the tensor product states $\times \omega$, where 
\begin{enumerate}
\item[(i)] If $\frac{\alpha}{2\pi}$ is irrational, the trace is the only spreadable state, i.e. $\omega$ is the Haar measure on $\bt$;
\item[(ii)] If $\frac{\alpha}{2\pi} = \frac{q_1^{n_1}\cdots q_s^{n_s}}{p_1^{m_1}\cdots p_r^{m_r}}$ is rational, $\omega$ is any state on $C(\mathbb{T})$ invariant under all rotations by $n_0$-th roots of unity, where $n_0 = p_1^{\{m_1/2\}}\cdots p_r^{\{m_r/2\}}$,
with $\{\frac{k}{2}\}$ is $\frac{k}{2}$ if $k$ is even and $\frac{k+1}{2}$ if $k$ is odd.
\end{enumerate}
\end{thm}
In the light of the result above, only the rational case needs a thorough discussion if tail algebras are to be dealt with. Indeed,
the tail algebra of the canonical trace is  trivial because the canonical trace satisfies the block-singleton condition.
Therefore, from now on we will suppose $\a=\frac{m}{n}$ with $m, n$ coprime.\\
 In this case, the grading
can more conveniently be described by the finite group $\bz_n=\{z\in\bt: z^n=1\}\subset\bt$ and by the bicharacter
$v: \widehat{\bz_n}\times \widehat{\bz_n}\rightarrow\bt$ given by
$v([l], [k])=e^{2\pi i \frac{m}{n}lk}$, where $\widehat{\bz_n}$ is identified with $\bz/n\bz$ (and $[\cdot]$ is the equivalence class mod $n$).\\
 For any spreadable state $\varphi$ on $\mathbb{A}_\a^\bn$,  we denote by
$\overline{\varphi}$ the average of $\varphi$ with respect to the action of the reduced gauge group $G=\bz_n$, that is

\begin{equation}
\overline{\varphi}=\frac{1}{|G|}\sum_{g\in G}\varphi\circ\a_{g}\,
\end{equation}
or, more concretely, $\overline{\varphi}=\frac{1}{n}\sum_{k=0}^{n-1}\varphi\circ\a_{z_0^k}$ with $z_0=e^{\frac{2\pi i}{n}}$.\\
Denote by $\varphi_{\bz}$ the unique spreadable extension of $\varphi$ to the infinite non-commutative 
torus indexed by the integers $\bz$, $\mathbb{A}_\a^\bz$. Note that $\varphi$  is dominated by $\overline{\varphi}$ and accordingly
$\varphi_\bz$ is still dominated by  $\overline{\varphi_\bz}$.
But then  the GNS representation of $\varphi_\bz$ can be realized as a subrepresentation of
the GNS representation of  $\overline{\varphi_\bz}$ , that is there exists a vector $x_\varphi$ in 
$\ch_{\overline{\varphi_\bz}}$
 such that
$\langle \pi_{\overline{\varphi_\bz}}(a) x_\varphi, x_\varphi \rangle=\varphi_\bz(a)$ for all $a$ in $\mathbb{A}_\a^\bz$, 
meaning the closed subspace $\ch_{\varphi_\bz}:=\overline{\pi_{\overline{\varphi_\bz}}(\mathbb{A}_\a^\bz)x_\varphi}$ is
a subrepresentation equivalent to the GNS representation of $\varphi_\bz$.\\
Set
$\mathcal{R}_{\varphi_\bz}:=\pi_{\varphi_\bz}(\mathbb{A}_\a^\bn)''$
and
$\mathcal{R}_{\overline{\varphi_\bz}}:=\pi_{\overline{\varphi_\bz}}(\mathbb{A}_\a^\bn)''$.
Let $\Psi:\mathcal{R}_{\overline{\varphi_\bz}}\rightarrow\mathcal{R}_{\varphi_\bz} $
the $*$-epimorphism obtained by restricting operators in the former algebra to the subspace
$\ch_{\varphi_\bz}$.
We recall that $\cs_{\pi_{\varphi_\bz}}^\tau$ is the convex set of all stationary states that are normal in the GNS representation of $\varphi_\bz$.

The following theorem completely describes tail algebras of the extreme states of the simplex of spreadable states on $\mathbb{A}_\a^\bn$, which we recall satisfy Property (ii) of Theorem \ref{ergohierspread}.
\begin{thm}\label{isotailtorus}
For any spreadable state $\varphi$  on $\mathbb{A}_\a^\bn$, one has  
$$\mathcal{T}_\varphi\cong \mathcal{T}_{\overline{\varphi}}\,,$$
with the $*$-isomorphism being realized by the restriction  to $\mathcal{T}_{\overline{\varphi}}$ of the  isomorphism $\Psi:\mathcal{R}_{\overline{\varphi_\bz}}\rightarrow\mathcal{R}_{\varphi_\bz}$ above.\\
Furthermore, if $\varphi$ is a tensor product state, then  $\cs^\tau_{\pi_{\varphi_\bz}}$ is the finite-dimensional Bauer simplex given by ${\rm conv}\{\varphi\circ\a_{g}: g\in \bz_n\}$ and 
$$\mathcal{T}_\varphi\cong \bc^d\quad d:=|\{\varphi\circ\a_g: g\in\bz_n\}|$$
in the $*$-isomorphism that identifies the minimal projections in $\bc^d$ with the supports of the extreme states of $\cs_{\pi_{\varphi_\bz}}^\tau$.
\end{thm}

\begin{proof}
We start by observing that, although $\varphi$ is not necessarily gauge invariant, the gauge action of
$\bz_n$ is all the same implemented in representation $\pi_{\varphi_{\bz}}$ on
$\pi_{\varphi_\bz}(\mathbb{A}_\a^\bn)$. Indeed, by the commutation rules of the non-commutative torus 
 $u_{-1} u_j u_{-1}^*= e^{2\pi i \frac{m}{n}}u_j$ for all $j\geq 0$, one sees that the adjoint action of $u_{-1}$
on $\mathbb{A}_\a^\bn$ coincides with $\a_{z_0^m}$ with $z_0=e^{2\pi i \frac{1}{n}}$.
Because $m$ and $n$ are coprime, by B\'ezout's identity there exist $k, l$ in $\bz$ such that $km+ln=1$. But then 
\begin{align*} 
 u_{-1}^k u_j{ (u_{-1})^*}^{k}= e^{2\pi i \frac{km}{n}}u_j=e^{2\pi i(\frac{1}{n}-l)}u_j=e^{2\pi i \frac{1}{n}}u_j\,,
\end{align*}
which means $u_{-1}^k$ implements the action of the generator of  the gauge group $\bz_n$, whose action is thus fully implemented.\\
The next step we take is to show that for all $g$ in $G=\bz_n$ the representations $\pi_{\varphi_\bz\circ\a_g}$ and
$\pi_{\varphi_\bz}$, understood as representation of $\mathbb{A}_\a^\bn$, are unitarily equivalent. This can be seen by noting that
$\pi_{\varphi_\bz\circ\a_g}$ is unitarily equivalent to $\pi_{\varphi_\bz}\circ\a_g$, and
each $\a_g$ is unitarily implemented in the GNS representation by the above discussion. In particular,
$\varphi\circ\a_g$ appears as a vector state (associated with  say $x_g$ in $\ch_{\varphi_\bz}$) in $\pi_{\varphi_\bz}$ (restricted to $\mathbb{A}_\a^\bn$).
As a result, the averaged state $\overline{\varphi}$ is a normal state in the representation $\pi_{\varphi_\bz}$ (restricted to $\mathbb{A}_\a^\bn$).\\
We now move on to show that the restriction map $\Psi$ defined above sends $\mathcal{T}_{\overline{\varphi}_\bz}$ onto
$\mathcal{T}_{\varphi_\bz}$. The inclusion $\Psi(\mathcal{T}_{\overline{\varphi}_\bz})\subseteq\mathcal{T}_{\varphi_\bz}$ is a straightforward consequence of the definition of tail algebra. As for the reverse inclusion, let $S$ be in $\mathcal{T}_{\varphi_\bz}$. By definition of $\mathcal{T}_{\varphi_\bz}$, for every $n$ in $\bn$ there exists $T_n$ in $\{ \pi_{\overline{\varphi_\bz}}(i_k(C(\bt))): k\geq n\}''$ such that $\Psi(T_n)=S$ and $\|T_n\|\leq \|S\|$ for all $n$
(that is because $\Psi$ sends $\{ \pi_{\overline{\varphi_\bz}}(i_k(C(\bt))): k\geq n\}''$ onto $\{ \pi_{\varphi_\bz}(i_k(C(\bt))): k\geq n\}''$ ).
Now the sequence $\{T_n: n\in\bn\}$ has a subsequence that converges to some $T$ in the weak operator topology. By construction $T$ sits in $\bigcap_n\{ \pi_{\overline{\varphi_\bz}}(\iota_k(\ga)): k\geq n\}''=\mathcal{T}_{\overline{\varphi}_\bz}$ and $\Psi(T)=S$ by continuity of $\Psi$ w.r.t. the weak operator topology.\\
As for injectivity of $\Psi$, let us 
now observe that $\overline{\varphi}_\bz(T)=\frac{1}{n}\sum_{g\in G} \langle T x_g, x_g \rangle$ holds for all
$T$ in $\mathcal{R}_{\overline{\varphi}}$, in that it certainly holds for all 
$T=\pi_{\overline{\varphi_\bz}}(a)$, $a$ in $\mathbb{A}_\a^\bn$.
Suppose now $\Psi(T^*T)=0$ for $T$ in $\mathcal{T}_{\overline{\varphi}_\bz}$. By the formula above
$\overline{\varphi}_\bz(T^*T)=\frac{1}{n}\sum_{g\in G} \|Tx_g\|^2=0$, which entails
$T^*T=0$ by  Lemma \ref{separating}.\\
As for the second part of the statement, we start by showing that $\mathcal{T}_\varphi$ is an abelian $C^*$-algebra.
This follows from the chain of $^*$-isomorphisms $\mathcal{T}_\varphi\cong\mathcal{T}_{\varphi_\bz}\cong\mathcal{T}_{\overline{\varphi_\bz}}\cong
E_{\overline{\varphi_\bz}} (\pi_{\overline{\varphi_\bz}}(\mathbb{A}_\a^\bn))''E_{\overline{\varphi_\bz}}$, where $E_{\overline{\varphi_\bz}}$ is the orthogonal projection onto $\ch_{\overline{\varphi_\bz}}^\tau$, and the last $*$-isomorphism can be proved in the exact same way as in the proof of Theorem \ref{01lawspread} and
relies on the fact that $\mathcal{T}_{\overline{\varphi}}$ is separated by $\xi_{\overline{\varphi}}$.
Finally, abelianness of algebras of the form $E_{\overline{\varphi_\bz}} (\pi_{\overline{\varphi_\bz}}(\mathbb{A}_\a^\bn))''E_{\overline{\varphi_\bz}}$ is shown in \cite[Proposition 4.1]{CDGR}.\\
Therefore, we need only show that $\mathcal{T}_\varphi$ is finite-dimensional with dimension given by the $d$ in the statement.
Now if $\varphi$ is extreme, then the states $\varphi\circ\a_g$ are extreme as well for all $g$ in $\bz_n$. Furthermore,
by what we saw above, they are normal in the representation $\pi_{\varphi_\bz}$, and thus by virtue of Theorem \ref{purerestriction} their restrictions to $\mathcal{T}_\varphi$ provide pure states of the tail algebra, hence 
${\rm dim}\,\mathcal{T}_\varphi \geq |\{\a\circ\a_g: g\in \bz_n\}|$. On the other hand, we must also have
${\rm dim}\,\mathcal{T}_\varphi \leq |\{\a\circ\a_g: g\in \bz_n\}|$ because $\{\varphi\circ\a_g: g\in \bz_n\}$ is a separating
set of states for the tail algebra (since their uniform convex combination gives $\overline{\varphi}$).
The rest of the statement is now a direct consequence of Corollary \ref{normalinvspread}.
\end{proof}

\begin{rem}\label{parafermions}
In \cite{BJLW} the authors address what they call braidable  states on the parafermion algebra, lavishing their attention to the corresponding  tail agebras.
Now the parafermion algebra is an instance of infinite twisted tensor product, with sample algebra $\bm_d(\bc)$ and grading group $\bz_d$.
As shown in \cite{GK08}, every braidable state is spreadable. However, it is presently unclear how large the class of braidable states on the parafermion algebra  actually is, whereas 
 the class of spreadable states forms a rich Bauer simplex whose extreme points are precisely the admissible tensor product states, see \cite[Section 4.2]{ADR}. The tail algebras of these extreme states can be described in the same way as
those for the infinite non-commutative torus by an analogue of Theorem \ref{isotailtorus}. Since the argument is entirely analogous to that for the infinite non-commutative torus, and both the statement and proof would require additional notation, we omit every detail.
\end{rem}

\begin{rem}\label{blocknoteqprod}
The result in \cite[Theorem 3.4]{ACGL} is inaccurate, for counterexamples can be given. To this end, we can consider the following example, which is actually  due to Elia Vincenzi and was contained in his Ph.D. thesis. However, since the
thesis is  no longer available online, we include a reasonably detailed account of that example.\\
Consider the $C^*$-dynamical system $(M_2(\bc),\bz_2\times \bz_2,\gamma)$, 
with gauge action $\gamma: \bz_2\times \bz_2\rightarrow$Aut$(M_2(\bc))$ determined by $\gamma_{(1,0)}(A)=ad_U(A)$, $\gamma_{(0,1)}(A)=ad_V(A)$ with
$$U=\begin{bmatrix}
1 & 0\\
0 & -1
\end{bmatrix}, 
V=\begin{bmatrix}
0 & 1\\
1 & 0
\end{bmatrix},
$$
together with the antisymmetric bicharater $v$ on $\widehat{\bz_2\times \bz_2}\cong \bz_2\times \bz_2$ given by 
$$v((x_1,x_2),(y_1,y_2))=(-1)^{x_1y_1+x_2y_2} \text{ for } (x_1,x_2),(y_1,y_2)\in \bz_2\times \bz_2.$$
Take the infinite twisted tensor product of $\bm_2(\bc)$ w.r.t. $v$, $\otimes^\bz_v M_2(\bc)$.
We note that the dynamics given by the permutation action gives rise to a $\bp_\bz$-abelian system, that is a system that satisfies the hypotheses of Theorem 3.4 in \cite{ACGL}.  This can be seen for instance by applying a minor variation of \cite[Theorem 4.10]{ADR}.
However, there are extreme exchangeable states on $\otimes^\bz_v M_2(\bc)$ that do not satisfy the block-singleton property, thus contradicting
the conclusion of Theorem 3.4 in \cite{ACGL}. Indeed, take $\omega$ to be the vector state on $M_2(\bc)$ induced by the vector $\frac{1}{\sqrt{2}}(1,i)\in\bc^2$. By Theorem \ref{prodstates}, there exists the tensor product state $\times \omega$ on $\otimes^\bz_v M_2(\bc)$, which is easily checked to be exchangeable, since the bicharacter $v$ is antisymmetric,  and not $\bz_2\times \bz_2$-invariant. It follows that $\times \omega$ does not satisfy the block singleton property, see Remark \ref{tailGinv}. In addition, $\times\om$ is an extreme  exchangeable state on $\otimes^\bz_v M_2(\bc)$ by the implication (ii)$\Rightarrow$(i) in the hierararchy of ergodicities for the exchangeable case, Theorem \ref{ergohierexchange}.

\end{rem}

\appendix
\section{}
The present appendix includes a series of results of various type which have been deferred to the end of the paper so as to
keep the general exposition as coincise as possible.
\subsection{Some equivalences of ergodic properties}

\begin{prop}\label{ergospread}
For a spreadable (or weakly spreadable)  state $\varphi$ on $\ast_\bn\ga$, the following conditions are equivalent:
\begin{itemize}
\item [(i)]  $\varphi$ is strongly clustering, {\it i.e.} $\lim_{n\rightarrow\infty} \varphi(x\tau^n(y))=\varphi(x)\varphi(y)$,  for  all $x, y$ in $\ast_\bn\ga$;
\item [(ii)] $E_\varphi$ is the projection onto $\bc\xi_\varphi$;
\item [(iii)] $\varphi(xy)=\varphi(x)\varphi(y)$,  for  all localized $x, y$ in $\ast_\bn\ga$ supported in subsets $F_x, F_y$ with either $F_x < F_y$ or $F_x > F_y$.
\end{itemize}
\end{prop}

\begin{proof}

We will  show the chain of implications  $(i)\Rightarrow (ii)\Rightarrow (iii)\Rightarrow (i)$.\\
We show that if (i) holds, then $E_\varphi \pi_\varphi(x)\xi_\varphi=\varphi(x) \xi_\varphi$ for all $x$ in $\ast_\bn \ga$, from which
(ii) follows. For any $y$ in $\ast_\bn\ga$ we have
\begin{align*}
&\langle E_\varphi \pi_\varphi(x)\xi_\varphi, \pi_\varphi(y)\xi_\varphi\rangle=\langle \pi_\varphi(x)\xi_\varphi, E_\varphi\pi_\varphi(y)\xi_\varphi\rangle=\lim_n \langle \pi_\varphi(x)\xi_\varphi, \frac{1}{n}\sum_{k=0}^n \pi_\varphi(\tau^k(y))\xi_\varphi\rangle\\
&=\lim_n \frac{1}{n}\sum_{k=0}^n\varphi(x\tau^k(y))= \varphi(x)\varphi(y)= \langle \varphi(x)\xi_\varphi,
\pi_\varphi(y)\xi_\varphi \rangle\, ,
\end{align*}
which proves the sought equality by cyclicity of $\xi_\varphi$.\\
Suppose now (ii) holds and let $x, y$ in $\ast_\bn\ga$ with the support of $y$ entirely to the right of the support of
$x$. We have
\begin{align*}
\varphi(xy)= \langle \pi_\varphi(xy)\xi_\varphi, \xi_\varphi \rangle= \langle \pi_\varphi(x)\pi_\varphi(\tau^q(y))\xi_\varphi, \xi_\varphi \rangle
\end{align*}
for all $q\geq0 $, thanks to 
$\bj_\bn$-invariance of $\varphi$ (or by weak spreadability), since there certainly exists an increasing map from from $\bn$ to $\bn$ that is the identity on the support of $x$ and shifts by $q$ that of $y$. 
But then
\begin{align*}
\varphi(xy)&=\langle \pi_\varphi(x)\frac{1}{n}\sum_{q=0}^{n-1}\pi_\varphi(\tau^q(y))\xi_\varphi, \xi_\varphi \rangle\\
&=\lim_n \langle \pi_\varphi(x)\frac{1}{n}\sum_{q=0}^{n-1}\pi_\varphi(\tau^q(y))\xi_\varphi, \xi_\varphi \rangle=\langle \pi_\varphi(x) E_\varphi \pi_\varphi(y)\xi_\varphi, \xi_\varphi\rangle\\
&= \langle  \pi_\varphi(x) \varphi(y)\xi_\varphi,\xi_\varphi \rangle = \varphi(x)\varphi(y)\,
\end{align*}
and we are done.
If it is the support of $x$ to lie to the right of the support of $y$, then
\begin{align*}
\langle \pi_\varphi(xy)\xi_\varphi, \xi_\varphi \rangle=\langle \xi_\varphi, \pi_\varphi(y^*)\pi_\varphi(x^*)\xi_\varphi \rangle=\langle \xi_\varphi, \pi_\varphi(y^*)\pi_\varphi(\tau^q(x^*))\xi_\varphi \rangle
\end{align*}
for the same reason as above. At this point we can safely proceed as above to find that 
$\varphi(xy)=\varphi(x)\varphi(y)$.\\
Finally, suppose $(iii)$ holds, that is  the factorization $\varphi(xy)=\varphi(x)\varphi(y)$ holds for all localized
$x, y$ with suitable supports. For $n$ large enough the support of $\tau^n(y)$ will lie to the right of the support of $x$, hence $\varphi(x\tau^n(y))=\varphi(x)\varphi(y)$, and 
$\lim_n \varphi(x\tau^n(y))=\varphi(x)\varphi(y)$. The general limit property is easily achieved by norm density of the set of localized elements.
\end{proof}

\subsection{An $L^2$ version of K\"{o}stler's ergodic Theorem}

We here derive an $L^2$ version of K\"{o}stler's ergodic Theorem \cite[Theorem 8.4]{K} where we remove the faithfulness assumption. This is the essential ingredient used to obtain the full product factorization condition of the conditional expectation onto the tail algebra for spreadable processes. 

Let us first fix some notation. For any $N\geq 1$, $\theta_N$ in $\bj_\bn$ is partial shift given by $\theta_N(n)=n+1$ for $n\geq N$ and $\theta_N(n)=n$ for $n< N$.
For  any fixed spreadable state $\varphi$, denote by  $V_{N}$ the isometry acting on the GNS Hilbert space $\ch_\varphi$ as
$V_{N} \pi_\varphi(x)\xi_\varphi:=\pi_\varphi (\theta_N(x))\xi_\varphi\, , \,x\in \ast_\bn\ga$.
Note that
\begin{equation}\label{commrul}
V_{N}V_{N-1}=V_{N-1}^2\, .
\end{equation}
Set
$$M_k^{(N)}:=\frac{1}{N}\sum_{l=0}^{N-1}V_{k}^l\quad T_N:=\prod_{k=0}^NV_{k}^{kN}M_k^{(N)}\, .$$
\begin{prop}\label{ergoKostler}
With the notation set above, for any spreadable state $\varphi$ on $\ast_\bn \ga$ one has
\begin{equation*}
\lim_{N\rightarrow\infty} T_N= E_\varphi
\end{equation*}
in the strong operator topology.
\end{prop}

\begin{proof}
Since $\|T_N\|\leq 1$ for all $N$ in $\bn$, it is enough to verify the limit equality in the statement only on the dense subspace of $\ch_\varphi$
 given by all vectors  $\pi_\varphi(x)\xi_\varphi$ with $x$ running through the (norm dense) subalgebra of localized elements.
Let $x$ sit in the subalgebra generated by the first $N_0$ embeddings and let $N\geq N_0$. From
$\theta_N(x)=x$, we have $V_{N}\pi_\varphi(x)\xi_\varphi=\pi_\varphi(x)\xi_\varphi$, hence
$M_N^{(n)}\pi_\varphi(x)\xi_\varphi=\pi_\varphi(x)\xi_\varphi $ for all $n$.\\
For every $k$, denote by $E_k$ the orthogonal projection onto the fixed-point subspace of $V_{k}$.
Note that $E_{0}$ is the projection $E_\varphi$ onto $\ch_\varphi^{\bj_\bn}=\ch_\varphi^\tau$.
 Now for all
$k\leq N$  we have $E_k V_{N}=E_k$. Indeed, proving this equality amounts to showing 
$V_{N}^*E_k=E_k$. This in turn is equivalent to $V_{N}E_k=E_k$, which is of course verified.
We clearly have $E_k V_{N}^l=E_k$ for all $l\geq 0$, and in particular $E_k M_N^{(n)} =E_k$.\\
But then we have
\begin{align*}
&T_N\pi_\varphi(x)\xi_\varphi=M_0^{(N)} V_{1}^{N} M_1^{(N)}\cdots V_{k}^{kN}M_k^{(N)}\cdots V_{N}^{N^2}M_N^{(N)}\pi_\varphi(x)\xi_\varphi=\\
&M_0^{(N)} V_{1}^{N} M_1^{(N)}\cdots V_{k}^{kN}M_k^{(N)}\cdots V_{N_0-1}^{(N_0-1)N}M_{N_0-1}^{(N)}\pi_\varphi(x)\xi_\varphi=\\
&\left(\prod_{l=0}^{N_0-1} V_{l}^{lN}(M_l^{(N)}-E_l)\right)E_{N_0}\pi_\varphi(x)\xi_\varphi+\\
&\left(\prod_{l=0}^{N_0-2} V_{l}^{lN}(M_l^{(N)}-E_l)\right)E_{N_0-1}\pi_\varphi(x)\xi_\varphi+\\
&\left(\prod_{l=0}^{N_0-3} V_{l}^{lN}(M_l^{(N)}-E_l)\right)E_{N_0-2}\pi_\varphi(x)\xi_\varphi+\\
&\ldots+\\
& (M_0^{(N)}-E_0) V_{1}(M_1^{(N)}-E_1)E_2\pi_\varphi(x)\xi_\varphi+\\
&(M_0^{(N)}-E_0)E_1\pi_\varphi(x)\xi_\varphi+\\
&E_0\pi_\varphi(x)\xi_\varphi\,.
\end{align*}
In particular, the sequence of vectors $\{T_N \pi_\varphi(x)\xi_\varphi-E_\varphi\pi_\varphi(x)\xi_\varphi\}_N$ is the sum of $N_0$
terms, each of which goes to zero thanks to the mean ergodic theorem.
For example, the first of these terms is the the sequence (in N) of vectors
$$\left(\prod_{l=0}^{N_0-1} V_{l}^{lN}(M_l^{(N)}-E_l)\right)E_{N_0}\pi_\varphi(x)\xi_\varphi$$
which go to zero by joint continuity on bounded subsets of $\cb(\ch_\varphi)$  of the product in the strong operator topology, in that
 each factor $(M_l^{(N)}-E_l)$ strongly converges to zero by virtue of the mean ergodic theorem.
\end{proof}

\begin{lem}\label{insertproj}
Let $\varphi$ be a spreadable state on $\ast_\bn\ga$. For any localized $a, b$ in $\ast_\bn\ga$ with disjoint supports one has
$$E_\varphi \pi_\varphi(ab)E_\varphi= E_\varphi \pi_\varphi(a) E_\varphi \pi_\varphi(b) E_\varphi\, .$$
\end{lem}

\begin{proof}

For any $N$-tuple $ \bold{l}_N=(l_1, l_2, \ldots, l_N)$ in $\{0, 1, \ldots N-1\}^N$ define
$\theta_{N, \bold{l}_N}:= \prod_{i=0}^N \theta_i^{iN+l_i}=\theta_0^{l_0}\theta_1^{N+l_1}\cdots \theta_N^{N^2+l_N}$.
As proved in \cite[Lemma 8.5]{K}, the maps $\theta_{N, \bold{l}_N}$ enjoy the following property: for any fixed $N$, 
$\bold{l}_N$, $\bold{k}_N$ one has $\theta_{N, \bold{l}_N}(i)<\theta_{N, \bold{k}_N}(j)$ if $i<j<N$.\\
 By its very definition, $T_N$ can also be written as $T_N=\frac{1}{N^{N+1}}\sum_{\bold{l}_N}V_{\theta_{N, \bold{l}_N}}$, where
$V_{\theta_{N, \bold{l}_N}}$ is the isometry on $\ch_\varphi$ defined by $V_{\theta_{N, \bold{l}_N}} \pi_\varphi(x)\xi_\varphi:=
\pi_\varphi(\theta_{N, \bold{l}_N}(x))\xi_\varphi$, for all $x$ in $\ast_\bn\ga$.\\
That said, we are ready to prove the statement. For all $\xi, \eta$ in  $\ch_\varphi$ we have
\begin{align*}
&\langle E_\varphi \pi_\varphi(a)E_\varphi \pi_\varphi(b) E_\varphi E_\varphi\xi , E_\varphi \eta \rangle\\
&=\lim_N
\langle  \frac{1}{N^{N+1}}\sum_{\bold{l}_N}V_{\theta_{N, \bold{l}_N}}\pi_\varphi(a)\frac{1}{N^{N+1}}\sum_{\bold{k}_N}V_{\theta_{N, \bold{k}_N}} \pi_\varphi(b) E_\varphi E_\varphi\xi , E_\varphi \eta \rangle\\
\end{align*}
by joint continuity of the product w.r.t. the strong operator topology on bounded sets of $\cb(\ch_\varphi)$.\\
The right-hand side of the equality above can in turn be expressed as follows:
\begin{align*}
&\lim_N \frac{1}{N^{2(N+1)}}\sum_{\bold{l}_N,\bold{k}_N}\langle \theta_{N, \bold{l}_N}\pi_\varphi(a) V_{\theta_{N, \bold{k}_N}} \pi_\varphi(b) E_\varphi E_\varphi\xi, E_\varphi\eta\rangle= \\
&\lim_N \frac{1}{N^{2(N+1)}}\sum_{\bold{l}_N,\bold{k}_N}\langle \pi_\varphi(a) \pi_\varphi(b) E_\varphi E_\varphi\xi, E_\varphi\eta\rangle=
\langle \pi_\varphi(a) \pi_\varphi(b) E_\varphi E_\varphi\xi, E_\varphi\eta\rangle=\\
&\langle E_\varphi \pi_\varphi(a)\pi_\varphi(b) E_\varphi E_\varphi\xi , E_\varphi \eta \rangle\,,
\end{align*}
where the first equality is got to by $\bj_\bn$-invariance of the bounded linear functional $\langle (\cdot) E_\varphi\xi,  E_\varphi\eta\rangle$
and the property enoyed  by the maps $\theta_{N, \bold{l}_N}$ recalled above.
\end{proof}

\section{}

 \begin{prop}\label{csupport}
Let $\mathcal{B}\subset\mathcal{A}$ be a unital inclusion of $C^*$-algebras. If $\varphi$  is a state on
  $\mathcal{A}$ with central support, then $\varphi\upharpoonright_{\mathcal{B}}$ has central support as well as a state on
  $\mathcal{B}$
   \end{prop}

\begin{proof}
Denote by $\mathcal{K}\subset\ch_\varphi$ the cyclic subspace generated by $\mathcal{B}$ in the GNS representation (of $\mathcal{A}$) of $\varphi$, that is $\ck:=\overline{\pi_\varphi(\mathcal{B})\xi_\varphi}$. By uniqueness of  the GNS representation, $\ck$ is (up to unitary equivalence) the GNS representation of $\varphi\upharpoonright_\mathcal{B}$. By von Neumann bicommutant theorem, the restriction map
$\pi_\varphi(\mathcal{B})\ni\pi_\varphi(b)\,\mapsto\pi_\varphi(b)\upharpoonright_\ck\, \in \pi_{\varphi\upharpoonright _\mathcal{B}}(\mathcal{B})$ can be extended to a surjective $*$-homomorphism $\Phi:\pi_\varphi(\mathcal{B})''\rightarrow\pi_{\varphi\upharpoonright _\mathcal{B}}(\mathcal{B})'' $.\\
Let now $T$ be in $\pi_{\varphi\upharpoonright _\mathcal{B}}(\mathcal{B})''$ such that $T \xi_\varphi=0$. Since 
$\Phi$ is surjective, there exists $S$ in $\pi_\varphi(\mathcal{B})''$ such that $\Phi(S)=T$. Because
$\Phi$ restricts operators to $K$, we also have $S\xi_\varphi=0$, hence $S=0$ as by hypothesis
$\xi_\varphi$ is separating for the whole $\pi_\varphi(\mathcal{A})''$, which finally yields $T=0$.
\end{proof}

\begin{lem}\label{separating}
For any $G$-invariant, shift-invariant state $\varphi$ on $\otimes^\bn_v \ga$, the GNS vector
$\xi_\varphi$ is separating for $\mathcal{T}_\varphi$.
\end{lem}

\begin{proof}

Since the state $\varphi$ is $G$-invariant, the action of $G$ is unitarily implemented on
$\ch_\varphi$ by a family of unitaries $\{V_g: g\in G\}$ uniquely determined by
$V_g\pi_\varphi(x)\xi_\varphi= \pi_\varphi(\gamma_g(x))\xi_\varphi$, for all $g$ in $G$, $x$ in $\otimes^\bn_v \ga$.\\\
We claim that $\mathcal{T}_\varphi^ G:=\{T\in \mathcal{T}_\varphi: V_gTV_g^*=T, g\in G\}$ is contained
in the center of the von Neumann algebra $\pi_\varphi(\otimes^\bn_v \ga)''$. From this
 it immediately follows that
$\xi_\varphi$ is separating for $\mathcal{T}_\varphi^G $ (as $\xi_\varphi$ is separating for the whole commutant $\pi_\varphi(\otimes^\bn_v \ga)'$).\\
Let us now denote by $E_G$ the canonical conditional expectation of $\pi_\varphi(\otimes^\bn_v \ga)''$ onto its $G$-fixed-point algebra, {\it i.e.} $E_G$ is obtained by averaging the adjoint action of the $V_g$'s.\\
If $T$ in  $\mathcal{T}_\varphi$ is such that $T\xi_\varphi=0$, then
$V_gT^*TV_g^*\xi_\varphi=0$ for all $g$ in $G$, hence $E_G(T^*T)\xi_\varphi=0$. 
Since $\xi_\varphi$ is separating for  $\mathcal{T}_\varphi^G$, we find $E_G(T^*T)=0$.
Because $E_G$ is faithful, we finally get to $T^*T=0$, hence $T=0$.\\
In order to end the proof, we need to make sure the claim made above holds true.
This can be seen by noting that any element in $\mathcal{T}_\varphi^G$ commutes with  the image under $\pi_\varphi$ of all
homogeneous localized elements of $\otimes^\bn_v \ga$.
\end{proof}

\section{}
We recall that if $f$ is a positive linear functional on a non-unital $C^*$-algebra $\ga$, then $\|f\|=\sup_\lambda f(e_\lambda)$, where
$\{e_\lambda: \lambda\in \Lambda\}$ is a contractive, positive approximate unit, see {\it e.g.} \cite[Proposition 2.1.5.]{D}
For any positive linear functional $f$ on $\ga$ with norm less than or equal to $1$, we denote by $\widetilde{f}$ its extension as a state of
$\widetilde{\ga}=\ga+\bc I$, the one-point unitalization of $\ga$. The GNS representation $\pi_f$ of $f$ is by definition the restriction to $\ga$ of the GNS representation
of $\widetilde{f}$ as a representation of $\widetilde{\ga}$ (note that in this way $\pi_f$ acts on $\ch_{\widetilde{f}}$). To ease the notation, the GNS vector $\xi_{\widetilde{f}}$
will be simply denoted by $\xi_f$.

\begin{lem}\label{approxunit}
Let $\varphi$ be a normalized positive linear functional on $\ast_\bn^0\ga$ such that  $\|\varphi\restriction_\cb\|=1$, where $\cb\subset\ast_\bn^0\ga$ is a $C^*$-subalgebra.
For any contractive, positive approximate unit $\{e_\lambda: \lambda\in\Lambda\}$ of $\cb$, one has
$$\lim_\lambda\pi_\varphi(e_\lambda)\xi_\varphi=\xi_\varphi\,, $$
where the convergence is understood in norm.
\end{lem}

\begin{proof}
It is a matter of computation. Indeed, one has
\begin{align*}
&0\leq \|\pi_\varphi(e_\lambda)\xi_\varphi-\xi_\varphi\|^2= \widetilde{\varphi}((e_\l-1)^*(e_\l -1))= \widetilde{\varphi}(e_\l^2)+1-2\widetilde{\varphi}(e_\l)\leq\\
&\widetilde{\varphi}(e_\l)+1-2\widetilde{\varphi}(e_\l)\rightarrow 0
\end{align*}
as $\varphi(e_\l)$ goes to $1$ because the norm of the restriction of $\varphi$ to $\cb$ is $1$ by hypothesis (in the inequality we have taken into account $e_\l^2\leq e_\l$, which holds as 
$0\leq e_\l\leq 1$).
\end{proof}

\section*{Acknowledgments}
\noindent
Both authors are members of INDAM-GNAMPA and acknowledge the support of INdAM-GNAMPA Project 2026 ``Simmetrie distribuzionali per processi stocastici quantistici", Project CUP\_UP E53C25002010001.


\end{document}